\documentclass[a4paper,11pt]{article}
\usepackage[latin1]{inputenc}
\usepackage[english]{babel}
\usepackage{amsmath}
\usepackage{amsfonts}
\usepackage{amssymb}
\usepackage{epsfig}
\usepackage{amsopn}
\usepackage{amsthm}
\usepackage{color}
\usepackage{graphicx}
\usepackage{enumerate}
\usepackage{mathrsfs}
\usepackage{cite}
\newtheorem{theorem}{Theorem}[section]
\newtheorem{corollary}[theorem]{Corollary}
\newtheorem{lemma}[theorem]{Lemma}
\newtheorem{proposition}[theorem]{Proposition}
\newtheorem{definition}[theorem]{Definition}
\newtheorem{remark}[theorem]{Remark}
\newtheorem*{theorem*}{Theorem}
\newtheorem*{lemma*}{Lemma}
\newtheorem*{remark*}{Remark}
\newtheorem*{definition*}{Definition}
\newtheorem*{proposition*}{Proposition}
\newtheorem*{corollary*}{Corollary}
\numberwithin{equation}{section}
\newcommand{\real}{\mathbb{R}}

\let\ced=\c         % cedilla
\def\e{\varepsilon}        % Also, \varepsilon
\def\qed{\,\unskip\kern 6pt \penalty 500
\raise -2pt\hbox{\vrule \vbox to8pt{\hrule width 6pt
\vfill\hrule}\vrule}\par}
\definecolor{darkblue}{rgb}{0.05, .05, .65}
\definecolor{darkgreen}{rgb}{0.1, .65, .1}
\definecolor{darkred}{rgb}{0.8,0,0}
\newcommand{\beqn}{\begin{equation}}
\newcommand{\eeqn}{\end{equation}}
\newcommand{\bear}{\begin{eqnarray}}
\newcommand{\eear}{\end{eqnarray}}
\newcommand{\bean}{\begin{eqnarray*}}
\newcommand{\eean}{\end{eqnarray*}}
\begin{document}
%%%%%%%%%%%%%%%%%%%%
%%%%%%%%%%%%%%%%%%%%

%%%%%%%%%%%%%%%%%%%%
%%%%%%%%%%%%%%%%%%%%
\title{\huge \bf Self-similar extinction for a diffusive Hamilton-Jacobi equation with critical absorption}

\author{
\Large Razvan Gabriel Iagar\,\footnote{Instituto de Ciencias
Matem\'aticas (ICMAT), Nicol\'as Cabrera 13-15, Campus de
Cantoblanco, 28049, Madrid, Spain, \textit{e-mail:}
razvan.iagar@icmat.es},\footnote{Institute of Mathematics of the
Romanian Academy, P.O. Box 1-764, RO-014700, Bucharest, Romania.}
\\[4pt] \Large Philippe Lauren\ced{c}ot\,\footnote{Institut de
Math\'ematiques de Toulouse, CNRS UMR~5219, Universit\'e de
Toulouse, F--31062 Toulouse Cedex 9, France. \textit{e-mail:}
Philippe.Laurencot@math.univ-toulouse.fr}\\ [4pt] }
\date{\today}
\maketitle
%%%%%%%%%%%%%%%%%%%%
%%%%%%%%%%%%%%%%%%%%

%%%%%%%%%%%%%%%%%%%%
%%%%%%%%%%%%%%%%%%%%
\begin{abstract}
The behavior near the extinction time is identified for non-negative solutions to the diffusive Hamilton-Jacobi equation with critical gradient absorption
$$
\partial_tu-\Delta_p u+|\nabla u|^{p-1}=0 \quad \hbox{in} \
(0,\infty)\times\real^N\ ,
$$
and fast diffusion $2N/(N+1)<p<2$. Given a non-negative and radially symmetric initial condition with a non-increasing profile which decays sufficiently fast as $|x|\to\infty$, it is shown that the corresponding solution $u$ to the above equation approaches a uniquely determined \emph{separate variable solution} of the form
$$
U(t,x)=(T_e-t)^{1/(2-p)}f_*(|x|), \quad (t,x)\in (0,T_e)\times \real^N\ ,
$$
as $t\to T_e$, where $T_e$ denotes the finite extinction time of $u$. A cornerstone of the convergence proof is an underlying variational structure of the equation. Also, the selected profile $f_*$ is the unique non-negative solution to a second order ordinary differential equation which decays exponentially at infinity. A complete classification of solutions to this equation is provided, thereby describing all separate
variable solutions of the original equation. One important difficulty in the uniqueness proof is
that no monotonicity argument seems to be available and it is overcome by the construction of an appropriate \emph{Pohozaev functional}.
\end{abstract}

%%%%%%%%%%%%%%%%%%%%
%%%%%%%%%%%%%%%%%%%%

\vspace{1.0 cm}

%%%%%%%%%%%%%%%%%%%%
%%%%%%%%%%%%%%%%%%%%

\medskip

%%%%%%%%%%%%%%%%%%%%

%%%%%%%%%%%%%%%%%%%%
%%%%%%%%%%%%%%%%%%%%
\section{Introduction}
%%%%%%%%%%%%%%%%%%%%
%%%%%%%%%%%%%%%%%%%%

We perform a detailed analysis of the phenomenon
of finite time extinction for the diffusive
Hamilton-Jacobi equation with critical absorption and singular diffusion:
\begin{equation}\label{eq1}
\partial_tu-\Delta_pu+|\nabla u|^{p-1}=0, \qquad  (t,x)\in (0,\infty)\times\real^N,
\end{equation}
supplemented with the initial condition
\begin{equation}\label{eq2}
u(0,x)=u_0(x), \quad x\in\real^N.
\end{equation}
Here, as usual,
$$
\Delta_p u := {\rm div}(|\nabla u|^{p-2}\nabla u)\ ,
$$
where the exponent $p$ is assumed to belong to the fast diffusion range
\begin{equation*}
p_c:=\frac{2N}{N+1}<p<2\ ,
\end{equation*}
and the initial condition $u_0$ is assumed throughout the paper to satisfy
\begin{equation}\label{hypIC}
u_0\in W^{1,\infty}(\real^N)\ , \qquad u_0\ge 0\ , \qquad u_0\not\equiv 0\ .
\end{equation}
The term \emph{critical} refers here to the peculiar choice of the
exponent $p-1$ for the gradient absorption, for which the
homogeneity of the absorption matches that of the diffusion. Observe
that $p-1\in (0,1)$ as $p\in (p_c,2)$.

Equation~\eqref{eq1} is actually a particular case of the diffusive Hamilton-Jacobi equation with gradient absorption
\begin{equation}\label{eqgen}
\partial_tu-\Delta_pu+|\nabla u|^{q}=0, \qquad  (t,x)\in (0,\infty)\times\real^N,
\end{equation}
and it is already known that extinction in finite time
takes place for non-negative solutions to \eqref{eqgen} with initial data decaying
sufficiently fast as $|x|\to\infty$, provided $q\in (0,p/2)$ and $p\in (p_c,2]$ \cite{BLS01, BLSS02, Gi05, IaLa12, ILSxx}. Note that this range includes the values of $p$ and $q=p-1$ considered in this paper. Let us recall that, by extinction in finite time we mean that there exists $T_e\in(0,\infty)$ such that $u(t,x)=0$ for any $x\in\real^N$ and $t\ge T_e$, but $\|u(t)\|_{\infty}>0$ for all $t\in (0,T_e)$. The time $T_e$ is usually referred to as \emph{the extinction time} of the solution $u$.

The main feature of \eqref{eqgen} is the \emph{competition} between the two terms involving the space variable, the diffusion $-\Delta_p u$ and the absorption $|\nabla u|^q$, and their influence on the dynamics. As the properties of the diffusion equation and of the Hamilton-Jacobi equation (without diffusion) are very different, it is an interesting task to study the effects of their merging in \eqref{eqgen}, leading to different types of behavior in dependence on the relation between the exponents $p$ and $q$.

The development of the mathematical theory for \eqref{eqgen} begun with the semilinear case $p=2$, where techniques coming from linear theory (such as representation formulas via convolution with the heat kernel) were applied in order to get estimates on
the solutions. The qualitative theory together with the large time behavior are now well understood. For exponents $q>1$, the problem
has been investigated in a series of works \cite{ATU04, BSW02, BKaL04, BL99, BVD13, BGK04, GL07, GGK03, Gi05}. In this range, the
diffusion has an important influence on the evolution: either completely dominating, when $q>q_*:=(N+2)/(N+1)$, leading to
asymptotic simplification, or having a similar effect to the Hamilton-Jacobi term for $q\in(1,q_*]$, leading to a resonant, logarithmic-type behavior for $q=q_*$ \cite{GL07}, or to a behavior driven by \emph{very singular solutions} for $q\in (1,q_*)$
\cite{BKaL04}. For exponents $q\in (0,1)$, a singular phenomenon, extinction in finite time, shows up \cite{BLS01, BLSS02, Gi05}, and a deeper study of
the extinction mechanism has been performed recently in \cite{ILSxx}. It is shown that, in this range $q\in (0,1)$, rather unusual phenomena such as \emph{instantaneous shrinking} (that is, the
support of $u(t)$ becomes compact for any $t>0$, even if $u_0(x)>0$ for any $x\in\real^N$) and \emph{single point extinction} take
place. However, a precise description of the behavior of solutions near the extinction time is still missing. Finally, the critical case $q=1$ seems to be currently out of reach, though optimal decay estimates as
$t\to\infty$ are established in \cite{BRV97}.

A natural extension of the theory is to consider nonlinear
generalizations of the Laplacian operator in the diffusion term, and
the quasilinear operator $\Delta_p$ is one of the obvious
candidates. In this case, the study proved to be more involved and
challenging, as the classical linear techniques do not work anymore.
Due to this difficulty, the qualitative theory for \eqref{eqgen}
with $p\neq 2$ has been understood quite recently. Our main interest
focuses on the fast/singular diffusion case $p\in (p_c,2)$, for
which the qualitative theory has been developed starting from
\cite{IaLa12}, where all exponents $q>0$ are considered. Two
critical exponents are identified in \cite{IaLa12}:
$q=q_*:=p-N/(N+1)$ and $q=p/2$. These critical values limit ranges
of different behaviors: diffusion dominates for $q>q_*$ leading to
asymptotic simplification, while there is a balance between
diffusion and absorption for $q\in [p/2,q_*]$. At last but not
least, finite time extinction occurs for $0<q<p/2$ as soon as the
initial condition $u_0$ decays sufficiently fast as $|x|\to\infty$.

Still the behavior of non-negative solutions to \eqref{eqgen} is not uniform within this range of values of $q$. Indeed, a by-product of the analysis in \cite{IaLa12, ILSxx} reveals that there is another critical exponent in $(0,p/2)$, namely $q=p-1$. In fact, though the driving mechanism of extinction at a global scale is the absorption term in \eqref{eqgen} for all $q\in (0,p/2)$, a fundamental difference in the occurrence of finite time extinction shows up within this range. As shown in \cite{ILSxx}, when $0<q<p-1$, a special phenomenon known as
\emph{instantaneous shrinking} takes place: for non-negative initial data decaying sufficiently fast at infinity, the support of $u(t)$ is compact for any $t>0$, even if $u_0(x)>0$ for any $x\in\real^N$. For a suitable class of radially symmetric initial conditions with a non-increasing profile, this property is enhanced by the dynamics and \emph{single point extinction} takes place: the support of $u(t)$ shrinks to the singleton $\{0\}$ as $t\to T_e$.

This is in \emph{sharp contrast} with the range $p-1\leq q<p/2$,
where, as shown in \cite[Proposition~4.4]{ILSxx}, \emph{simultaneous
extinction} occurs: that is, $u(t,x)>0$ for any $t\in (0,T_e)$ and
$x\in\real^N$. This positivity property up to the extinction time is
clearly due to the diffusion term which is thus not completely
negligible in this range: some kind of balance between the two terms
is expected, at least for initial data $u_0$ rapidly decaying as
$|x|\to\infty$. In view of this analysis, the exponent $q=p-1$ to
which we devote this work acts as an interface between the two
different extinction mechanisms and describing it fully adds up to
the general understanding of the dynamics of \eqref{eqgen}.

Let us point out that we restrict our analysis to the range $p>p_c$ in which there is a competition between the diffusion term aiming at positivity in the whole space $\real^N$ and the gradient absorption term being the driving mechanism of extinction. We do not consider here the limiting case $p=p_c$ which is more involved. We also leave aside the subcritical range $p\in (1,p_c)$ as finite time extinction also occurs for the fast diffusion equation without the gradient absorption term. We rather expect a different kind of competition in that range, namely between two extinction mechanisms stemming from both diffusion and absorption.

We finally mention that the phenomenon of extinction in finite time and the dynamics close to the extinction time have been an object of study also for other models exhibiting competition between diffusion and absorption, see \cite{Be01, BHV01, BeSh07, Ka87, FV, HV92} for instance. A well-studied example is the fast diffusion equation with zero order absorption
\begin{equation}\label{PMEa}
\partial_t u-\Delta u^m+u^q=0, \qquad (t,x)\in (0,\infty)\times \real^N\ ,
\end{equation}
in the range of parameters $m_c:=(N-2)_+/N<m< 1$ and $0<q<1$, see for example \cite{Ka87, FV, FGV, dPSa02} and the references therein. For
this equation, a threshold between single point extinction and simultaneous extinction also appears but at the exponent $q=m$. Further
comments on similarities and differences between \eqref{eq1} and \eqref{PMEa} with $0<m=q<1$ will be provided in the next section.

%%%%%%%%%%%%%%%%%%%%
%%%%%%%%%%%%%%%%%%%%
\section{Main results}
%%%%%%%%%%%%%%%%%%%%
%%%%%%%%%%%%%%%%%%%%

We are now ready to present the main contributions of this paper.
First, as already pointed out, a consequence of the specific choice
of the exponent $p-1$ for the gradient absorption in \eqref{eq1} is
the homogeneity of $-\Delta_p u + |\nabla u|^{p-1}$, which allows us
to look for \emph{separate variable solutions}, that is, particular
solutions $U$ to \eqref{eq1} having the following form:
\begin{equation}\label{sepvar}
U(t,x)=T(t)f(|x|), \qquad (t,x)\in (0,\infty)\times\real^N.
\end{equation}
It is readily seen that, by direct calculation, we have
$$
T(t)=[(2-p)(T_0-t)_+]^{1/(2-p)}, \qquad t\in (0,\infty),
$$
for some $T_0>0$, and $f$ is a non-negative solution to the following second order ordinary differential equation:
\begin{equation}\label{SSODE}
\left(|f'|^{p-2}f'\right)'(r)+\frac{N-1}{r}\left(|f'|^{p-2}f'\right)(r)+f(r)-|f'(r)|^{p-1}=0, \qquad r\in (0,\infty)\ ,
\end{equation}
with the usual notation $r=|x|$ that will be used
throughout the paper. The expected regularity of $U$ entails that $f'(0)=0$ while the value of $f(0)$ is still unknown at this point. We are thus led to study the following question: for which initial conditions
\begin{equation}\label{ICa}
f(0)=a>0, \quad f'(0)=0,
\end{equation}
is the solution $f=f(\cdot;a)$ to \eqref{SSODE}-\eqref{ICa} non-negative for all $r\ge 0$? The answer is given by our first result.

%%%%%%%%%%%%%%%%%%%%
%%%%%%%%%%%%%%%%%%%%
\begin{theorem}[Classification of profiles]\label{th.uniq}
Given $a>0$ there is a unique solution $f(\cdot;a)\in C^1([0,\infty))$ to \eqref{SSODE}-\eqref{ICa} such that $(|f'|^{p-2}f')(\cdot;a)\in C^1([0,\infty))$.

Furthermore, there exists a unique $a_*>0$ such that $f(r;a_*)>0$ for all $r\ge 0$ and there is $c_*>0$ such that
\begin{equation}\label{decay}
f(r;a_*)\sim c_* r^{-(N-1)/(p-1)}e^{-r/(p-1)} \qquad {\rm as} \ r\to\infty.
\end{equation}
Moreover, a complete classification of the behavior of the solutions $f(\cdot;a)$ to \eqref{SSODE}-\eqref{ICa} is available:

\noindent (a) For $a\in(0,a_*)$, there holds $f(r;a)>0$ for all $r\ge 0$ and
$$
f(r;a)\sim \left( \frac{2-p}{p-1}
r \right)^{-(p-1)/(2-p)} \qquad {\rm as} \ r\to\infty.
$$

\noindent (b) For $a\in(a_*,\infty)$, there exists
$R(a)\in(0,\infty)$ such that $f(r;a)>0$ for $r\in(0,R(a))$, $f(R(a);a)=0$, and $f'(R(a);a)<0$. In particular, $f(\cdot;a)$ changes sign in $(0,\infty)$.
\end{theorem}
%%%%%%%%%%%%%%%%%%%%
%%%%%%%%%%%%%%%%%%%%

Before going forward to state our main convergence (or stabilization near extinction) result, let us make several remarks and comments on the previous result.

\medskip

Let us first point out that Theorem~\ref{th.uniq} indicates that \eqref{SSODE} has several \emph{ground states}, that is, non-negative $C^1$-smooth solutions which decay to zero as $r\to\infty$. Indeed, for any $a\in (0,a_*]$, $f(\cdot;a)$ is a ground state according to the previous definition. This is in sharp contrast with equations of the form
$$
(|w'|^{p-2} w')'(r) + \frac{N-1}{r} (|w'|^{p-2} w')(r) - w(r)^\alpha + w(r)^\beta = 0\ , \qquad r\in (0,\infty)\ ,
$$
with $p\in (1,N)$ and $0<\alpha<\beta$, for which a single ground state exists, see \cite{Kw89, ST00, Ya91a, Ya91b,  ShWa13, ShWa16} and the references therein. The uniqueness statement in Theorem~\ref{th.uniq} thus amounts to select among the ground states of \eqref{SSODE} the one with the fastest decay. This requires to identify the possible decay rates for ground states at infinity and thus a refined analysis is needed.

We next notice that the behavior of $f(r;a_*)$ as $r\to\infty$
depends on the dimension $N$. Roughly speaking, this dependence is
due to the $r$-dependent term $(N-1) (|f'|^{p-2} f')(r)/r$ in
\eqref{SSODE} and the proof of Theorem~\ref{th.uniq} is much more
involved for $N\ge 2$ than for $N=1$. In fact, since the variable
$r$ does not appear explicitly in \eqref{SSODE} when $N=1$, we
introduce in the companion paper \cite{IL16} a transformation which
maps solutions of \eqref{SSODE} onto the solutions to a first order
nonlinear ordinary differential equation. Besides being simpler to
study, solutions of the latter equation also enjoy a monotonicity
property with respect to the parameter $a$ and this is a key feature
to establish the uniqueness of $a_*$. We refer to \cite{IL16} for
the proof of Theorem~\ref{th.uniq} when $N=1$.

Such a transformation does not seem to be available for $N\ge 2$ and, moreover, it seems that no
monotonicity of $f(\cdot;a)$, or other functions associated to it, with respect to $a>0$ holds true. This is thus an extra difficulty to be overcome on the way to the proof of Theorem~\ref{th.uniq}. Indeed, several recent uniqueness or classification results for ordinary differential equations derived for the study of self-similar profiles for nonlinear diffusion equations heavily rely on the monotonicity of the solutions with respect to the shooting parameter, see \cite{CQW03, Shi04, IaLa13a, IaLa13b, YeYi15} for example. The situation therein is similar to the one encountered here and there are several ground states, the monotonicity property being very helpful to select the one with the fastest decay and establish its uniqueness.

We failed to find such a property for \eqref{SSODE} in dimensions $N\geq2$ and we thus use some different, and technically more involved, ideas. The uniqueness proof is actually based on the construction of a \emph{Pohozaev functional} associated to the differential equation
$$
g''(r)+\left(1+\frac{N-1}{r}\right)g'(r)+g(r)^{1/(p-1)}-\frac{N-1}{r^2}g(r)=0,
$$
where $g:=-|f'|^{p-2}f'$. This approach of constructing Pohozaev-type functionals to study uniqueness for positive radial solutions to some elliptic equations seems to come, up to our knowledge, from Yanagida \cite{Ya91a, Ya91b}, while closer references to our case are the recent works by Shioji and Watanabe \cite{ShWa13, ShWa16}.

\medskip

Thanks to Theorem~\ref{th.uniq} we are in a position to state the main convergence result of the present paper. Its validity requires further assumptions on the initial condition $u_0$: we assume that
\begin{equation}
u_0 \;\text{ is radially symmetric and }\; \nabla u_0(x)\cdot x \le 0\ , \qquad x\in\real^N\ , \label{radsymdec}
\end{equation}
along with an exponential decay at infinity. More
precisely, there is $\kappa_0>0$ such that
\begin{equation}\label{decayin}
0 \le u_0(x) \le \kappa_0\ e^{-|x|/(p-1)}\ , \qquad
x\in\mathbb{R}^N\ .
\end{equation}

%%%%%%%%%%%%%%%%%%%%
%%%%%%%%%%%%%%%%%%%%
\begin{theorem}[Convergence near extinction]\label{th.conv}
Let $u$ be the solution to the Cauchy problem \eqref{eq1}-\eqref{eq2} with an initial condition $u_0$ satisfying \eqref{hypIC}, \eqref{radsymdec}, and \eqref{decayin}. We denote the finite extinction time of $u$ by $T_e \in (0,\infty)$. Then
\begin{equation}\label{eq.conv}
\lim\limits_{t\to T_e}(T_e-t)^{-1/(2-p)} \|u(t)-U_*(t)\|_\infty=0,
\end{equation}
where
$$
U_*(t,x)=[(2-p)(T_e-t)]^{1/(2-p)}f(|x|;a_*), \qquad (t,x)\in (0,T_e)\times \real^N\ ,
$$
and $f(\cdot;a_*)$ is defined in Theorem~\ref{th.uniq}.
\end{theorem}
%%%%%%%%%%%%%%%%%%%%
%%%%%%%%%%%%%%%%%%%%

The proof of Theorem~\ref{th.conv} combines
several different techniques: sharp estimates for $u$ and $\nabla u$ near the extinction time, identification of the $\omega$-limit set in self-similar variables, and the uniqueness of the fast decaying solution to \eqref{SSODE} given by Theorem~\ref{th.uniq}. The first two steps actually rely on a \emph{variational structure} of Eq.~\eqref{eq1} which is only available for non-negative, radially symmetric solutions with non-increasing profiles and was noticed in \cite{LStxx} for a related problem. More precisely, \eqref{eq1} is in that
framework  a \emph{gradient flow} in $L^2(\real^N,e^{|x|}dx)$ for the functional
$$
\mathcal{J}(v)=\frac{1}{p}\int_{\real^N} e^{|x|}\ |\nabla v(x)|^{p}\,dx.
$$
This structure not only allows us to adapt a technique from \cite{BeHo80} (used originally for the fast diffusion equation $\partial_t \phi - \Delta \phi^m=0$ in a bounded domain with homogeneous Dirichlet boundary conditions and $m\in (0,1)$) to
derive optimal bounds near the extinction time, but also persists in self-similar variables and ensures that the $\omega$-limit set with respect to these variables only contains stationary solutions enjoying suitable integrability properties. Combining these
information with the uniqueness provided by Theorem~\ref{th.uniq} we conclude that $f(\cdot;a_*)$ is the only element of the $\omega$-limit set, which completes the proof.

\medskip

As already mentioned in the Introduction, the behavior near the extinction time has been studied also for the fast diffusion equation with zero order absorption \eqref{PMEa}. In particular, there is a corresponding critical case for \eqref{PMEa}, namely $q=m$, studied in \cite[Sections~4-5]{FV} in dimension $N=1$ and in \cite{dPSa02} for $N\geq2$. There is a \emph{striking difference}
between \eqref{eq1} and \eqref{PMEa} with $q=m$: for the latter there is a \emph{unique radially symmetric ground state} to the associated
ordinary differential equation, while the
convergence to zero \emph{no longer guarantees uniqueness} in \eqref{SSODE}-\eqref{ICa}, as we have seen in Theorem~\ref{th.uniq}. The approaches developed in \cite{FV,dPSa02} thus do not seem to be applicable to \eqref{eq1} and the study of \eqref{eq1} requires new and more involved arguments to select the attracting profile among the different existing ones. Another noticeable difference is that, for \eqref{PMEa} with $q=m$, it is sufficient that $u_0(x) \to 0$ as $|x|\to\infty$ in order for a convergence result similar to Theorem~\ref{th.conv} to hold true (at least in dimension $N=1$, see \cite[Theorem~5.1]{FV}). But, conditions on the tail of $u_0(x)$ as $|x|\to\infty$ are needed for Theorem~\ref{th.conv} to be valid. Such limitations on $u_0$ are actually necessary, since a close inspection of the proof of \cite[Theorem~1.2]{ILSxx} shows that, if
$$
\lim\limits_{|x|\to\infty} |x|^{(p-1)/(2-p)}u_0(x)=\infty \;\;\text{ and }\;\; u_0(x)>0 \;\;\text{for any}\;\; x\in\real^N,
$$
then extinction in finite time does not even occur and the corresponding solution to \eqref{eq1} remains positive in $\real^N$ for any $t>0$. However, in the light of the above result, it is likely that the required exponential decay \eqref{decayin} might be relaxed.

\medskip

\noindent \textbf{Organization of the paper.} The proofs of the main
results are given in two separate sections. Taking first
Theorem~\ref{th.uniq} for granted, we devote the next
Section~\ref{sec.s3} to the dynamical behavior of \eqref{eq1}. There
we exploit the variational structure of Eq.~\eqref{eq1} to derive
temporal estimates for $u(t)$ and $\nabla u(t)$ for $t\in(0,T_e)$
and prove that, in self-similar variables, the $\omega$-limit set
only contains non-zero solutions to \eqref{SSODE}-\eqref{ICa} which
belong to $W^{1,\infty}(\real^N)\cap L^2(\real^N;e^{|x|}dx)$. The
fact that there is a unique solution to \eqref{SSODE} enjoying such
a decay property is a consequence of  Theorem~\ref{th.uniq} which is
proved in Section~\ref{sec.s4}. Its proof relies on ordinary
differential equations techniques and in particular the study of
several auxiliary functions.

%%%%%%%%%%%%%%%%%%%%
%%%%%%%%%%%%%%%%%%%%
\section{Dynamical behavior}\label{sec.s3}
%%%%%%%%%%%%%%%%%%%%
%%%%%%%%%%%%%%%%%%%%

Consider $u_0\in W^{1,\infty}(\mathbb{R}^N)$ satisfying \eqref{hypIC}, \eqref{radsymdec}, and \eqref{decayin}. According to \cite[Theorems~1.1--1.3]{IaLa12} and \cite[Proposition~4.4]{ILSxx} there is a unique non-negative (viscosity) solution $u\in C([0,\infty)\times \mathbb{R}^N)$ to \eqref{eq1}-\eqref{eq2} which is also a weak solution and enjoys the following properties: there
are $T_e>0$ and $C_0=C_0(N,p,u_0)$ such that
\begin{equation}
u(t,x)>0 \;\mbox{ for }\; (t,x)\in (0,T_e)\times\mathbb{R}^N\ ,
\qquad u(t,x)=0 \;\mbox{ for }\; (t,x)\in
[T_e,\infty)\times\mathbb{R}^N\ , \label{db1}
\end{equation}
and
\begin{equation}
|\nabla\log u(t,x) | \le C_0 \left( 1 + \|u_0\|_\infty^{(2-p)/p}
t^{-1/p}\right)\ , \qquad (t,x)\in (0,T_e)\times\mathbb{R}^N\ .
\label{db2}
\end{equation}
Furthermore, the rotational invariance of \eqref{eq1}, the comparison principle, and the symmetry and monotonicity properties
of $u_0$ entail that, for each $t\in (0,T_e)$,
\begin{equation}
x\mapsto u(t,x) \;\mbox{ is radially symmetric with a non-increasing
profile}, \label{db3}
\end{equation}
see, e.g., \cite[Lemma~2.2]{ILSxx} for a proof.

%%%%%%%%%%%%%%%%%%%%
%%%%%%%%%%%%%%%%%%%%
\subsection{Temporal estimates: upper bounds}\label{s3.1}
%%%%%%%%%%%%%%%%%%%%
%%%%%%%%%%%%%%%%%%%%

We first prove that the spatial exponential decay of $u_0$ is preserved throughout time evolution.

%%%%%%%%%%%%%%%%%%%%
\begin{lemma}\label{lemA11}
For each $t\in (0,T_e)$, there holds
\begin{equation}
0 \le u(t,x) \le \kappa_0\ e^{-|x|/(p-1)}\ , \qquad
x\in\mathbb{R}^N\ . \label{db4}
\end{equation}
\end{lemma}
%%%%%%%%%%%%%%%%%%%%

\begin{proof} Setting $\Sigma(x) := \kappa_0\ e^{-|x|/(p-1)}$ for $x\in \mathbb{R}^N$ we observe that
$$
\partial_r\Sigma(r) = - \frac{\kappa_0}{p-1} e^{-r/(p-1)}\ , \qquad \partial_r^2\Sigma(r) = \frac{\kappa_0}{(p-1)^2} e^{-r/(p-1)}
$$
for $r=|x|\ne 0$ so that
\begin{align*}
-\Delta_p \Sigma(x) + |\nabla\Sigma(x)|^{p-1} & = \left( \frac{\kappa_0}{p-1} \right)^{p-1} \left( - 1 + \frac{N-1}{|x|} \right) e^{-|x|} + \left( \frac{\kappa_0}{p-1} \right)^{p-1} e^{-|x|} \\
& = \left( \frac{\kappa_0}{p-1} \right)^{p-1} \frac{N-1}{[x|}
e^{-|x|} \ge 0\ , \qquad x\ne 0\ .
\end{align*}

Now, if $\varphi\in C^2((0,T_e)\times\mathbb{R}^N)$ is an admissible
test function in the sense of \cite[Definition~2.3]{OhSa97} and
$\Sigma-\varphi$ has a local minimum at $(t_0,x_0)\in
(0,T_e)\times\mathbb{R}^N$, then the following alternative occurs:
either $x_0\ne 0$ which implies that $\nabla\varphi(t_0,x_0) =
\nabla\Sigma(x_0)\ne 0$ and $\partial_t\varphi(t_0,x_0) = 0$. We
then infer from the ellipticity of the $p$-Laplacian that
$$
\partial_t\varphi(t_0,x_0) - \Delta_p\varphi(t_0,x_0) + |\nabla\varphi(t_0,x_0)|^{p-1} \ge -\Delta_p \Sigma(x_0) + |\nabla\Sigma(x_0)|^{p-1} \ge 0\ ,
$$
which is the usual condition to be a supersolution. Or $x_0=0$ and
there is $\delta>0$ such that
$$
\varphi(t_0,x) - \varphi(t_0,0) \le \Sigma(x) - \Sigma(0) \le 0
\;\;\mbox{ and }\;\; \varphi(t,0) \le \varphi(t_0,0)
$$
for $t\in (t_0-\delta,t_0+\delta)$ and $x\in B(0,\delta)$. In
particular, $x\mapsto \varphi(t_0,x)$ has a local maximum at $x=0$
and $t\mapsto \varphi(t,0)$ has a local maximum at $t_0$. Therefore
$\nabla\varphi(t_0,0)=0$ and $\partial_t\varphi(t_0,0)=0$ and the
requirements to be a supersolution stated in
\cite[Definition~2.4]{OhSa97} are satisfied.

We have thus shown that $\Sigma$ is a supersolution to \eqref{eq1}
and it follows from \eqref{decayin} and the comparison principle
\cite[Theorem~3.9]{OhSa97} that $u(t,x)\le \Sigma(x)$ for $(t,x)\in
(0,T_e)\times\mathbb{R}^N$.
\end{proof}

The next step is to derive optimal temporal decay estimates for $u$.
To this end, we exploit the already mentioned variational structure
of \eqref{eq1} which is available here thanks to the symmetry and
monotonicity properties of of $u$ and adapt the technique of
\cite{BeHo80, SaVe94} to relate a weighted $L^2$-norm of $u$ and a
weighted $L^p$-norm of $\nabla u$ and obtain bounds on these
quantities. More precisely, define
\begin{equation}
\mathcal{I}(z) := \frac{1}{2} \int_{\mathbb{R}^N} e^{|x|} |z(x)|^2\
dx\ , \qquad \mathcal{J}(z) := \frac{1}{p} \int_{\mathbb{R}^N}
e^{|x|} |\nabla z(x)|^p\ dx\ , \label{db5}
\end{equation}
whenever it makes sense. We gather in the next lemma useful information on $\mathcal{I}(u)$ and $\mathcal{J}(u)$.

%%%%%%%%%%%%%%%%%%%%
\begin{lemma}\label{lemA120}
For each $t\in [0,T_e)$, both $\mathcal{I}(u(t))$ and
$\mathcal{J}(u(t))$ are positive and finite. In addition,
\begin{equation}
\mathcal{I}(u)\in C([0,T_e])\cap C^1((0,T_e])\ , \qquad \mathcal{J}(u)\in C((0,T_e])\ , \label{z0}
\end{equation}
and
\begin{align}
\frac{d}{dt}\mathcal{I}(u(t)) & = - p \mathcal{J}(u(t))\ , \qquad t\in (0,T_e)\ , \label{z1} \\
\mathcal{J}(u(t_2)) + \int_{t_1}^{t_2} \mathcal{D}(u(t))\ dt & \le \mathcal{J}(u(t_1))\ , \qquad 0 < t_1 < t_2 \le T_e\ , \label{z2}
\end{align}
where
\begin{equation}
\mathcal{D}(u(t)) := \int_{\mathbb{R}^N} e^{|x|} |\partial_t u(t,x)|^2\ dx\ , \qquad t\in (0,T_e)\ . \label{z3}
\end{equation}
\end{lemma}
%%%%%%%%%%%%%%%%%%%%

Formally the proof of Lemma~\ref{lemA120} readily follows from \eqref{eq1} after multiplication by $e^{|x|} u$ and $e^{|x|} \partial_t u$ and integration over $\mathbb{R}^N$, using repeatedly the property $x \cdot \nabla u(t,x) \le 0$ for $(t,x)\in (0,T_e)\times\mathbb{R}^N$. Due to the degeneracy of the diffusion term and the unboundedness of the weight $e^{|x|}$, a rigorous proof requires additional approximation and truncation arguments. In order not to delay further the proof of Theorem~\ref{th.conv}, we postpone the proof of Lemma~\ref{lemA120} to the end of this section.

Thanks to \eqref{z1} and \eqref{z2}, optimal estimates for $\mathcal{I}(u(t))$ and $\mathcal{J}(u(t))$ may now be derived by adapting an argument from \cite{BeHo80}.

%%%%%%%%%%%%%%%%%%%%
\begin{proposition}\label{lemA12}
There is $C_1=C_1(N,p,u_0)>0$ such that
\begin{equation}
\mathcal{I}(u(t))^{1/2} + \left[ t \mathcal{J}(u(t)) \right]^{1/p} \le C_1\
(T_e-t)^{1/(2-p)}\ , \qquad t\in (0,T_e]\ . \label{db6}
\end{equation}
\end{proposition}
%%%%%%%%%%%%%%%%%%%%

\begin{proof}
First, by \eqref{eq1} and the monotonicity and symmetry properties of $u$,
\begin{align*}
p \mathcal{J}(u(t)) & = \int_{\mathbb{R}^N} e^{|x|} |\nabla u(t,x)|^{p-2} \nabla u(t,x) \cdot \nabla u(t,x)\ dx \\
& = - \int_{\mathbb{R}^N} e^{|x|} \left( \Delta_p u(t,x) + |\nabla u(t,x)|^{p-2} \nabla u(t,x) \cdot \frac{x}{|x|} \right) u(t,x)\ dx \\
& = - \int_{\mathbb{R}^N} e^{|x|} u(t,x) \partial_t u(t,x)\ dx\ ,
\end{align*}
and we deduce from H\"older's inequality that
\begin{equation}
p \mathcal{J}(u(t))  \le \sqrt{2 \mathcal{I}(u(t))}
\sqrt{\mathcal{D}(u(t))}\ , \qquad t\in (0,T_e)\ . \label{db12}
\end{equation}

Now, fix $T\in (T_e/2,T_e)$. For $h\in (0,(T_e-T)/2)$ and $t\in [0,T]$,
define
$$
\mathcal{J}_h(u(t)) := \frac{1}{h} \int_t^{t+h} \mathcal{J}(u(s))\
ds\ .
$$
According to \eqref{z2} and \eqref{db12} there holds
$$
\frac{d}{dt} \mathcal{J}_h(u(t)) \le - \frac{1}{h} \int_t^{t+h}
\mathcal{D}(u(s))\ ds \le - \frac{p^2}{2h} \int_t^{t+h}
\frac{\mathcal{J}(u(s))^2}{\mathcal{I}(u(s))}\ ds
$$
for $t\in [0,T]$. We then deduce from \eqref{z1} and the above
inequality that
\begin{align*}
\frac{d}{dt} \left( \frac{\mathcal{J}_h(u(t))}{\mathcal{I}(u(t))^{p/2}} \right) & = \frac{1}{\mathcal{I}(u(t))^{p/2}} \frac{d\mathcal{J}_h(u(t))}{dt} - \frac{p}{2} \frac{\mathcal{J}_h(u(t))}{\mathcal{I}(u(t))^{(p+2)/2}} \frac{d\mathcal{I}(u(t))}{dt} \\
& \le \frac{p^2}{2} \frac{\mathcal{J}_h(u(t)) \mathcal{J}(u(t))}{\mathcal{I}(u)^{(p+2)/2}} - \frac{p^2}{2h \mathcal{I}(u(t))^{p/2}} \int_t^{t+h} \frac{\mathcal{J}(u(s))^2}{\mathcal{I}(u(s))}\ ds \\
& \le \frac{p^2}{2h \mathcal{I}(u(t))^{p/2}} \int_t^{t+h} \left(
\frac{\mathcal{J}(u(t))}{\mathcal{I}(u(t))} -
\frac{\mathcal{J}(u(s))}{\mathcal{I}(u(s))} \right)
\mathcal{J}(u(s)) \ ds\ .
\end{align*}
Owing to the continuity \eqref{z0} and the positivity of $\mathcal{I}(u)$ and $\mathcal{J}(u)$ in $[T_e/2,(T+T_e)/2]$, the
modulus of continuity
$$
\omega(h,(T+T_e)/2) := \sup \left\{ \left|
\frac{\mathcal{J}(u(t))}{\mathcal{I}(u(t))} -
\frac{\mathcal{J}(u(s))}{\mathcal{I}(u(s))} \right|\ :\ T_e/2\le s \le t
\le (T+T_e)/2\ , \ |t-s|\le h \right\}
$$
is well-defined and converges to zero as $h\to 0$. Consequently,
using also the time monotonicity of $\mathcal{I}(u)$ and
$\mathcal{J}(u)$, we obtain
\begin{align*}
\frac{d}{dt} \left( \frac{\mathcal{J}_h(u(t))}{\mathcal{I}(u(t))^{p/2}} \right) & \le \frac{p^2}{2h \mathcal{I}(u(t))^{p/2}} \omega(h,(T+T_e)/2) \int_t^{t+h} \mathcal{J}(u(s)) \ ds \\
& \le \frac{p^2}{2} \frac{\mathcal{J}(u_0)}{\mathcal{I}(u(T))^{p/2}}
\omega(h,(T+T_e)/2)
\end{align*}
for $t\in (T_e/2,T)$, hence, after integration,
$$
\frac{\mathcal{J}_h(u(t_2))}{\mathcal{I}(u(t_2))^{p/2}} \le
\frac{\mathcal{J}_h(u(t_1))}{\mathcal{I}(u(t_1))^{p/2}} +
\frac{p^2}{2} \frac{\mathcal{J}(u_0)}{\mathcal{I}(u(T))^{p/2}}
\omega(h,(T+T_e)/2) (t_2-t_1)
$$
for $T_e/2\le t_1 \le t_2 \le T$. Since the second term of the
right-hand side of the above inequality vanishes as $h\to 0$, we may
let $h\to 0$ and deduce from the time continuity of $\mathcal{J}(u)$
that $\mathcal{J}(u)/\mathcal{I}(u)^{p/2}$ is non-increasing in
$[T_e/2,T]$. As $T$ is arbitrary in $(T_e/2,T_e)$ we conclude that
\begin{equation*}
\frac{\mathcal{J}(u(t_2))}{\mathcal{I}(u(t_2))^{p/2}} \le
\frac{\mathcal{J}(u(t_1))}{\mathcal{I}(u(t_1))^{p/2}}\ , \qquad T_e/2
\le t_1 \le t_2 \le T_e\ .
\end{equation*}
In particular
\begin{equation}
\frac{\mathcal{J}(u(t))}{\mathcal{I}(u(t))^{p/2}} \le r_0 :=
\frac{\mathcal{J}(u(T_e/2))}{\mathcal{I}(u(T_e/2))^{p/2}} \ , \qquad T_e/2 \le t
\le T_e\ . \label{db13}
\end{equation}
We then infer from \eqref{z1} and \eqref{db13} that
$$
\frac{d}{dt} \mathcal{I}(u) \ge - p r_0 \mathcal{I}(u)^{p/2} \;\;\text{ in }\;\; [T_e/2,T_e]\ .
$$
Since $p<2$, we further obtain
$$
\frac{d}{dt} \mathcal{I}(u)^{(2-p)/2} \ge - \frac{p(2-p)}{2} r_0 \;\;\text{ in }\;\; [T_e/2,T_e]\ ,
$$
hence, after integrating over $(t,T_e)$, $t\in [T_e/2,T_e)$, and using the property
$\mathcal{I}(u(T_e))^{(2-p)/2}=0$,
$$
\mathcal{I}(u(t))^{(2-p)/2} \le \frac{p(2-p)}{2} r_0 (T_e-t)\ ,
\qquad t\in [T_e/2,T_e]\ .
$$
Combining the previous inequality with \eqref{db13} implies \eqref{db6} for $t\in [T_e/2,T_e]$. To complete the proof, we observe that, for $t\in (0,T_e/2)$,
$$
\mathcal{I}(u(t)) \le \mathcal{I}(u_0) \;\;\text{ and }\;\; t \mathcal{J}(u(t)) \le \int_0^t \mathcal{J}(u(s))\ ds \le \mathcal{I}(u_0)
$$
by Lemma~\ref{lemA120} and $t \mapsto \mathcal{I}(u_0)^\alpha (T_e-t)^{-1/(2-p)}$ is clearly bounded in $[0,T_e/2]$ for $\alpha\in \{1/2,1/p\}$.
\end{proof}

We next extend the temporal decay estimates established in Proposition~\ref{lemA12} to the $W^{1,\infty}$-norm of $u$.

%%%%%%%%%%%%%%%%%%%%
\begin{corollary}\label{corA13}
There is $C_2=C_2(N,p,u_0)>0$ such that
\begin{eqnarray}
\| u(t) \|_\infty & \le & C_2\ (T_e-t)^{1/(2-p)}\ t^{-N/2p}\ , \qquad t\in (0,T_e)\ , \label{db14} \\
\| \nabla u(t) \|_\infty & \le & C_2\ (T_e-t)^{1/(2-p)}\
t^{-(N+2)/2p}\ , \qquad t\in (0,T_e)\ . \label{db15} \end{eqnarray}
\end{corollary}
%%%%%%%%%%%%%%%%%%%%

\begin{proof}
Let $t\in (0,T_e)$. Observing that $2\mathcal{I}(u(t)) \ge \|u(t)\|_2^2$, we infer from the Gagliardo-Nirenberg inequality, \eqref{db2}, and \eqref{db6} that
\begin{align*}
\|u(t)\|_\infty & \le C \|\nabla u(t)\|_\infty^{N/(N+2)}
\|u(t)\|_2^{2/(N+2)} \\&
 \le C C_0^{N/(N+2)} \left( 1 + \|u_0\|_\infty t^{-1/p} \right)^{N/(N+2)} \|u(t)\|_\infty^{N/(N+2)} \left( 2\mathcal{I}(u(t)) \right)^{1/(N+2)} \\
 & \le C t^{-N/p(N+2)} \|u(t)\|_\infty^{N/(N+2)} (T_e-t)^{2/(2-p)(N+2)}\ ,
\end{align*}
from which we deduce \eqref{db14}. The bound \eqref{db15} is next a straightforward consequence of  \eqref{db2} and \eqref{db14}.
\end{proof}

%%%%%%%%%%%%%%%%%%%%
%%%%%%%%%%%%%%%%%%%%
\subsection{Temporal estimates: lower bound}\label{s3.15}
%%%%%%%%%%%%%%%%%%%%
%%%%%%%%%%%%%%%%%%%%

We now derive a lower bound for $\|u(t)\|_\infty$ which is of the same order as \eqref{db14}, thus revealing that we have identified the extinction rate.

%%%%%%%%%%%%%%%%%%%%
\begin{proposition}\label{propZ1}
There is a positive constant $C_3=C_3(N,p,u_0)$ such that
$$
\|u(t)\|_\infty \ge C_3 (T_e-t)^{1/(2-p)}\ , \qquad t\in (0,T_e)\ .
$$
\end{proposition}
%%%%%%%%%%%%%%%%%%%%

\begin{proof}
We first recall that, by \cite[Theorem~1.7]{IaLa12}, there is a positive constant $C$ depending only on $N$ and $p$ such that
\begin{equation}
\|\nabla u(t)\|_\infty \le C \|u(s)\|_\infty^{1/(p-1)} (t-s)^{-1/(p-1)}\ , \qquad 0 \le s < t\ . \label{y1}
\end{equation}
On the one hand we infer from \eqref{y1} and the Gagliardo-Nirenberg inequality that, for $s\ge 0$ and $t>s$,
\begin{align}
\|u(t)\|_\infty^{N+1} & \le C \|\nabla u(t)\|_\infty^N \|u(t)\|_1\ , \nonumber \\
\|u(t)\|_\infty^{N+1} & \le C \|u(s)\|_\infty^{N/(p-1)} (t-s)^{-N/(p-1)} \|u(t)\|_1\ . \label{y2}
\end{align}
On the other hand we fix $\vartheta\in (0,1]$ such that   $\vartheta<N(2-p)/(p-1)$ and deduce from H\"older's inequality that, for $t\ge 0$,
\begin{align*}
\|u(t)\|_1 & \le \| u(t)\|_\infty^{1-\vartheta} \int_{\real^N} u(t,x)^\vartheta\ dx \\
& \le \| u(t)\|_\infty^{1-\vartheta} \int_{\real^N} \left[ e^{\vartheta |x|/2} u(t,x)^\vartheta e^{-\vartheta|x|/2} \right] \ dx \\
& \le \| u(t)\|_\infty^{1-\vartheta} (2\mathcal{I}(u(t)))^{\vartheta/2} \left( \int_{\real^N} e^{-\vartheta |x|/(2-\vartheta)} \ dx \right)^{(2-\vartheta)/2} \\
& \le C(\vartheta) \| u(t)\|_\infty^{1-\vartheta} \mathcal{I}(u(t))^{\vartheta/2}\ ,
\end{align*}
hence, by Proposition~\ref{lemA12},
\begin{equation}
\|u(t)\|_1 \le C(\vartheta) \| u(t)\|_\infty^{1-\vartheta} (T_e-t)^{\vartheta/(2-p)}\ . \label{y3}
\end{equation}
Combining \eqref{y2} and \eqref{y3} gives
$$
\| u(t)\|_\infty^{N+\vartheta} \le C(\vartheta) \|u(s)\|_\infty^{N/(p-1)} (t-s)^{-N/(p-1)} (T_e-t)^{\vartheta/(2-p)}
$$
for $s\ge 0$ and $t>s$. Now, for $s\in (0,T_e)$, we set
$$
\tau(s) := \int_s^{T_e} \|u(t)\|_\infty^{(p-1)(N+\vartheta)/2N}\ dt\ ,
$$
and infer from the previous inequality that
$$
\tau(s) \le C(\vartheta) \left( \int_s^{T_e} (t-s)^{-1/2} (T_e-t)^{\vartheta(p-1)/2N(2-p)}\ dt \right) \left( - \tau'(s) \right)^{N/(p-1)(N+\vartheta)}\ .
$$
Since $t>s$ there holds $T_e-s>T_e-t$ and we obtain
\begin{align*}
\tau(s) & \le C(\vartheta) \left( \int_s^{T_e} (t-s)^{-1/2} \ dt \right) (T_e-s)^{\vartheta(p-1)/2N(2-p)} \left( - \tau'(s) \right)^{N/(p-1)(N+\vartheta)} \\
& \le C(\vartheta) (T_e-s)^{[N(2-p)+\vartheta(p-1)]/2N(2-p)} \left( - \tau'(s) \right)^{N/(p-1)(N+\vartheta)}\ .
\end{align*}
Consequently,
$$
\tau'(s) \tau(s)^{-(p-1)(N+\vartheta)/N} + C(\vartheta) (T_e-s)^{-[(p-1)(N+\vartheta)(N(2-p)+\vartheta(p-1))]/2 N^2 (2-p)} \le 0\ .
$$
Owing to the choice of $\vartheta$ there holds
$$
\frac{(p-1)(N+\vartheta)}{N}<1
$$
and
$$
\frac{(p-1)(N+\vartheta)[N(2-p)+\vartheta(p-1)]}{2 N^2 (2-p)} = \frac{(p-1)(N+\vartheta)}{N} \left( \frac{1}{2} + \frac{(p-1)\vartheta}{2N(2-p)} \right) < 1 \ .
$$
We may thus integrate the previous differential inequality over $(s,T_e)$, $s\in (0,T_e)$, and find, since $\tau(T_e)=0$,
$$
-\tau(s)^{[N(2-p)-\vartheta(p-1)]/N} + C(\vartheta) (T_e-s)^{[(N(2-p)-\vartheta(p-1))(2N(2-p)+(N+\vartheta)(p-1))]/2 N^2 (2-p)} \le 0\ ,
$$
or equivalently
$$
C(\vartheta) (T_e-s)^{[2N(2-p)+(N+\vartheta)(p-1)]/2N(2-p)} \le \tau(s)\ .
$$
Recalling that $t\mapsto \|u(t)\|_\infty$ is non-increasing by the comparison principle, we obtain
$$
\tau(s) \le (T_e-s) \|u(s)\|_\infty^{(p-1)(N+\vartheta)/2N}\ ,
$$
which gives, together with the previous lower bound on $\tau(s)$,
$$
C(\vartheta) (T_e-s)^{(p-1)(N+\vartheta)/2N(2-p)} \le \|u(s)\|_\infty^{(p-1)(N+\vartheta)/2N}\ ,
$$
thereby ending the proof.
\end{proof}

%%%%%%%%%%%%%%%%%%%%
%%%%%%%%%%%%%%%%%%%%
\subsection{Self-similar variables}\label{s3.2}
%%%%%%%%%%%%%%%%%%%%
%%%%%%%%%%%%%%%%%%%%

To study the convergence of $u$ to self-similarity we use the
classical transformation in self-similar variables and introduce the
function $v$ defined by
\begin{equation}
u(t,x) = \left( (2-p)(T_e-t) \right)^{1/(2-p)} v \left( -
\frac{1}{2-p} \log{\left( \frac{T_e-t}{T_e} \right)} , x \right) \label{db16}
\end{equation}
for $(t,x)\in (0,T_e)\times\mathbb{R}^N$ as well as the new time variable
$$
s := - \frac{1}{2-p} \log{\left( \frac{T_e-t}{T_e} \right)} \in
(0,\infty)\ .
$$
By \eqref{eq1}-\eqref{eq2}, the function $v$ solves
\begin{eqnarray}
\partial_s v - \Delta_p v + |\nabla v|^{p-1} - v & = & 0 \ , \qquad (s,x)\in (0,\infty)\times \mathbb{R}^N\ , \label{rdhjca} \\
v(0) & = & v_0 := \left( (2-p)T_e \right)^{1/(2-p)} u_0\ , \qquad
x\in\mathbb{R}^N\ . \label{rdhjcain}
\end{eqnarray}

We first translate the outcome of Section~\ref{s3.1} in terms of $v$ and deduce from Proposition~\ref{lemA12}, Corollary~\ref{corA13}, Proposition~\ref{propZ1}, and \eqref{db16} the following estimates.

%%%%%%%%%%%%%%%%%%%%
\begin{lemma}\label{lemA14}
There is $C_4=C_4(N,p,u_0)>1$ such that, for $s\in (0,\infty)$,
\begin{eqnarray}
\mathcal{I}(v(s)) + \left( 1 - e^{-(2-p)s} \right) \mathcal{J}(v(s)) & \le & C_4\ , \label{db17} \\
\left( 1 - e^{-(2-p)s} \right)^{N/2p} \|v(s)\|_\infty + \left( 1 -
e^{-(2-p)s} \right)^{(N+2)/2p} \|\nabla v(s)\|_\infty & \le & C_4\ ,
\label{db18}
\end{eqnarray}
and
\begin{equation}
\|v(s)\|_\infty \ge \frac{1}{C_4}\ . \label{db175}
\end{equation}
\end{lemma}
%%%%%%%%%%%%%%%%%%%%

We next exploit the variational structure of \eqref{rdhjca} in the
radially symmetric and monotone setting and investigate the behavior
of the energy
\begin{equation}
\mathcal{E} := \mathcal{J} - \mathcal{I}\ , \label{db19}
\end{equation}
along the trajectory $\{v(s)\ :\ s\ge 0\}$, recalling that
$\mathcal{I}$ and $\mathcal{J}$ are defined in \eqref{db5}.

%%%%%%%%%%%%%%%%%%%%
\begin{lemma}\label{lemA15}
The energy $s\mapsto \mathcal{E}(v(s))$ is a non-negative and
non-increasing function in $(0,\infty)$. In addition,
\begin{equation}
\int_0^\infty \int_{\mathbb{R}^N} e^{|x|} |\partial_s v(s,x)|^2\
dxds < \infty\ . \label{db20}
\end{equation}
\end{lemma}
%%%%%%%%%%%%%%%%%%%%

\begin{proof}
Let us begin with the monotonicity of the energy. We only give a formal proof below as it can be justified rigorously by arguing as in the proof of Lemma~\ref{lemA120}, see Section~\ref{s3.4}. Let
$s>0$. We infer from \eqref{rdhjca} and the monotonicity and symmetry properties of $v$ that
\begin{align*}
\frac{d}{ds} \mathcal{E}(v(s)) & = - \int_{\mathbb{R}^N} \left[ \mathrm{div}\left( e^{|x|} |\nabla v(s,x)|^{p-2} \nabla v(s,x) \right) + e^{|x|} v(s,x) \right] \partial_s v(s,x)\ dx \\
& = - \int_{\mathbb{R}^N} e^{|x|} \left[ \Delta_p v(s,x) + |\nabla v(s,x)|^{p-2} \nabla v(s,x) \cdot \frac{x}{|x|} + v(s,x) \right] \partial_s v(s,x)\ dx \\
& = - \int_{\mathbb{R}^N} e^{|x|} |\partial_s v(s,x)|^2\ dx\ ,
\end{align*}
from which the time monotonicity of the energy follows. Also,
\begin{equation}
\int_0^S \int_{\mathbb{R}^N} e^{|x|} |\partial_s v(s,x)|^2\ dxds \le
\mathcal{E}(v_0) - \mathcal{E}(v(S))\ , \qquad S>0\ . \label{db21}
\end{equation}

To establish the non-negativity of the energy, we adapt an argument
from \cite[Lemma~4.1]{dPSa02} and assume for contradiction that
there is $s_0>0$ such that $\mathcal{E}(v(s_0))<0$. On the one hand,
thanks to the just established monotonicity of the energy, we deduce
that
\begin{equation}
\mathcal{E}(v(s))<0\ , \qquad s\ge s_0\ . \label{db22}
\end{equation}
On the other hand, it follows from \eqref{rdhjca} and the
monotonicity and symmetry properties of $v$ that
\begin{align*}
\frac{d}{ds} \mathcal{I}(v(s)) & = -\int_{\mathbb{R}^N} e^{|x|} \left[ |\nabla v(s,x)|^p + |\nabla v(s,x)|^{p-2} \nabla v(s,x) \cdot \frac{x}{|x|} \right] v(s,x)\ dx \\
& \qquad + \int_{\mathbb{R}^N} e^{|x|} \left[ v(s,x) - |\nabla v(s,x)|^{p-1} \right] v(s,x)\ dx \\
& = -p \mathcal{J}(v(s)) + 2 \mathcal{I}(v(s)) \\
& = -p \mathcal{E}(v(s)) + \frac{2-p}{2} \mathcal{I}(v(s))\ .
\end{align*}
Recalling \eqref{db22} we infer from the above differential
inequality that
$$
\frac{d}{ds} \mathcal{I}(v(s)) \ge \frac{2-p}{2} \mathcal{I}(v(s)) +
p |\mathcal{E}(v(s_0))|\ , \qquad s\ge s_0\ ,
$$
which contradicts the boundedness \eqref{db17} of $\mathcal{I}(v)$.
Therefore, $\mathcal{E}(v(s))\ge 0$ for all $s\ge 0$. Furthermore,
the right-hand side of \eqref{db21} is bounded from above by
$\mathcal{E}(v_0)$, from which we deduce \eqref{db20}.
\end{proof}

%%%%%%%%%%%%%%%%%%%%
%%%%%%%%%%%%%%%%%%%%
\subsection{Convergence}\label{s3.3}
%%%%%%%%%%%%%%%%%%%%
%%%%%%%%%%%%%%%%%%%%

Thanks to the outcome of the previous two sections we are now in a
position to complete the proof of Theorem~\ref{th.conv}. The first
step is to show that the trajectory $\{v(s)\ :\ s\ge 0\}$ is compact
in a suitable space and that the $\omega$-limit set $\omega(v_0)$ in
that space contains only stationary solutions to \eqref{rdhjca}.
More precisely, we define $\omega(v_0)$ as follows:

%%%%%%%%%%%%%%%%%%%%
\begin{definition}\label{definOL}
$w_*\in
\omega(v_0)$ if and only if
\begin{itemize}
\item[(a)] $w_*$ is a non-negative radially symmetric function in $W^{1,\infty}(\mathbb{R}^N)\cap L^2(\mathbb{R}^N;e^{|x|}\ dx)$ with a non-increasing profile and $\nabla w_*\in L^p(\mathbb{R}^N;e^{|x|}\ dx)$,
\item[(b)] and there is a sequence $(s_k)_{k\ge 1}$ of positive real numbers such that
\begin{equation}
\lim_{k\to\infty} s_k = \infty \;\;\text{ and }\;\;
\lim_{k\to\infty} \|v(s_k)-w_*\|_\infty= 0\ . \label{db23}
\end{equation}
\end{itemize}
\end{definition}
%%%%%%%%%%%%%%%%%%%%

We first study the compactness properties of the trajectory $\{v(s)\
:\ s\ge 0\}$. As usual when dealing with unbounded domains, we begin
with a control of the behavior as $|x|\to\infty$.

%%%%%%%%%%%%%%%%%%%%
\begin{lemma}\label{lemA16}
There is $C_5=C_5(N,p)$ such that, if $w$ is a radially symmetric
function in $L^2(\mathbb{R}^N;e^{|x|}\ dx)$ with $\nabla w \in
L^p(\mathbb{R}^N;e^{|x|}\ dx)$, then
\begin{eqnarray}
|w(x)| & \le & C_5\ |x|^{-(N-1)/p} e^{-|x|/p} \mathcal{J}(w)^{1/p}\
, \qquad x\in \mathbb{R}^N\ , \label{db24} \\ \int_{\{|x|\ge R\}}
e^{|x|} |w(x)|^2\ dx & \le & C_5\ \ell(R) \mathcal{J}(w)^{1/p}\ ,
\qquad R>0\ , \label{db25}
\end{eqnarray}
where
$$
\ell(R) := \int_R ^\infty r^{-(2-p)(N-1)/p} e^{-(2-p)r/p}\ dr\ ,
\qquad R>0\ .
$$
\end{lemma}
%%%%%%%%%%%%%%%%%%%%

\begin{proof}
Let $x\in\mathbb{R}^N$, $x\ne 0$, and define $\varrho(r) := r^{N-1}
e^r$ for $r>0$. Since $w$ is radially symmetric, we infer from
H\"older's inequality that
\begin{align*}
|w(x)| & = \left| \int_{|x|}^\infty \partial_r w(r)\ dr \right| \le \int_{|x|}^\infty \left| \partial_r w(r)\right|\ dr  \\
& \le \left( \int_{|x|}^\infty \varrho(r)^{-1/(p-1)}\ dr \right)^{(p-1)/p} \left( \int_{|x|}^\infty \varrho(r) |\partial_r w(r)|^p\ dr \right)^{1/p} \\
& \le C \varrho(|x|)^{-1/p} \mathcal{J}(w)^{1/p}\ ,
\end{align*}
hence \eqref{db24}. The estimate \eqref{db25} is next a straightforward consequence of \eqref{db24}.
\end{proof}

%%%%%%%%%%%%%%%%%%%%
\begin{lemma}\label{lemA17}
The trajectory $\{v(s)\ :\ s\ge 1\}$ is relatively compact in
$C(\mathbb{R}^N)$ and the set $\omega(v_0)$ is non-empty.
\end{lemma}
%%%%%%%%%%%%%%%%%%%%

\begin{proof}
By Lemma~\ref{lemA14}, the trajectory $\{v(s)\ :\ s\ge 1\}$ is
bounded in $W^{1,\infty}(\mathbb{R}^N)$ and the Ascoli-Arzel\`a
theorem implies that it is relatively compact in $C(\bar{B}(0,R))$
for any $R>0$. Moreover,
\begin{equation}
|v(s,x)| \le C_5 C_4^{1/p} |x|^{-(N-1)/p} e^{-|x|/p}\ , \qquad s\ge
1\ , \;\; x\ne 0\ , \label{db26}
\end{equation}
by Lemma~\ref{lemA16}, from which the relative compactness in $C(\real^N)$ follows.
Therefore there are a sequence $(s_k)_{k\ge 1}$ and $w_*\in
C(\mathbb{R}^N)$ such that \eqref{db23} holds true. Combining the
properties of $(v(s_k))_{k\ge 1}$ with Lemma~\ref{lemA14}, we
readily conclude that $w_*$ is a non-negative radially symmetric
function in $W^{1,\infty}(\mathbb{R}^N)\cap
L^2(\mathbb{R}^N;e^{|x|}\ dx)$ with a non-increasing profile and
$\nabla w_*\in L^p(\mathbb{R}^N;e^{|x|}\ dx)$, that is, $w_*\in
\omega(v_0)$.
\end{proof}

We now exploit Lemma~\ref{lemA15} to identify the elements in
$\omega(v_0)$.

%%%%%%%%%%%%%%%%%%%%
\begin{proposition}\label{prA18}
If $w_*\in \omega(v_0)$ then it is a weak solution to
\begin{equation}
- \Delta_p w_* + |\nabla w_*|^{p-1} - w_* = 0 \;\;\text{ in }\;\;
\mathbb{R}^N\ . \label{db27}
\end{equation}
\end{proposition}
%%%%%%%%%%%%%%%%%%%%

\begin{proof}
Consider $w_*\in \omega(v_0)$. According to the definition of
$\omega(v_0)$ there is a sequence $(s_k)_{k\ge 1}$ of positive real
numbers such that \eqref{db23} holds true. Assuming without loss of
generality that $s_k> 2$ for all $k\ge 1$, we set
$$
v_k(s,x) := v(s_k+s,x)\ , \qquad (s,x)\in (-1,1)\times\mathbb{R}^N\
, \quad k\ge 1\ .
$$
We infer from Lemma~\ref{lemA14} that $(v_k)_{k\ge 1}$ is bounded in
$L^\infty(-1,1;W^{1,\infty}(\mathbb{R}^N))$, while
Lemma~\ref{lemA15} ensures that
\begin{equation}
\int_{-1}^1 \int_{\mathbb{R}^N} e^{|x|} |\partial_s v_k(s,x)|^2\
dxds \le \int_{-1+s_k}^{1+s_k} \int_{\mathbb{R}^N} e^{|x|}
|\partial_s v(s,x)|^2\ dxds \mathop{\longrightarrow}_{k\to\infty} 0\
. \label{db28}
\end{equation}
In particular, $(\partial_s v_k)_{k\ge 1}$ is bounded in
$L^2((-1,1)\times\mathbb{R}^N)$ and we infer from
\cite[Corollary~4]{Si87} and the compactness of the embedding of
$W^{1,\infty}(B(0,n))$ in $C(\bar{B}(0,n))$ for all $n\ge 1$ that
there exist a subsequence of $(s_k)_{k\ge 1}$ (not relabeled) and
$w\in C([-1,1]\times \mathbb{R}^N)$ such that
\begin{equation*}
\lim_{k\to\infty} \sup_{(s,x)\in [-1,1]\times \bar{B}(0,n)} |
v_k(s,x) - w(s,x)| = 0\ , \qquad n\ge 1\ .
\end{equation*}
Together with \eqref{db26}, this convergence leads us to:
\begin{equation}
\lim_{k\to\infty} \sup_{s\in [-1,1]} \| v_k(s) - w(s)\|_\infty = 0\
. \label{db29}
\end{equation}
Furthermore, $w\in L^\infty(-1,1;W^{1,\infty}(\mathbb{R}^N))$ by
\eqref{db18}.

We next deduce from \eqref{db28} and H\"older's inequality that, for
$s\in (-1,1)$,
$$
\|v_k(s)-v_k(0)\|_2 \le \sqrt{|s|}\ \left( \int_{-1}^1
\int_{\mathbb{R}^N} |\partial_s v_k(\sigma,x)|^2\ dxd\sigma
\right)^{1/2} \mathop{\longrightarrow}_{k\to\infty} 0\ .
$$
Combining this property with \eqref{db23}, \eqref{db26}, and
\eqref{db29} gives
\begin{equation}
w(s) = w_*\ , \qquad s\in [-1,1]\ . \label{db30}
\end{equation}
Finally, observing that $v_k$ is a solution in $(-1,1)\times
\mathbb{R}^N$ to \eqref{rdhjca} with initial condition $v(-1+s_k)$,
we infer from \eqref{db18}, \cite{DBFr85b}, and \cite[Theorem~1.1]{DBFr85a} that there is $\alpha\in (0,1)$ depending on $N$,
$p$, and $C_4$ such that $(\nabla v_k)_{k\ge 1}$ is bounded in
$C^{\alpha/2,\alpha}([-1/2,1/2]\times B(0,n))$ for all $n\ge 1$.
Owing to the Ascoli-Arzel\`a theorem and \eqref{db29} we may thus
assume further that $(\nabla v_k)_{k\ge 1}$ converges towards
$\nabla w_*$ in $C([-1/2,1/2]\times \bar{B}(0,n))$ for all $n\ge 1$.
We are then in a position to pass to the limit in the equation
solved by $v_k$ (which is nothing but \eqref{rdhjca}) and conclude
that $w_*$ is a weak solution to \eqref{db27}.
\end{proof}

\medskip

\begin{proof}[Proof of Theorem~\ref{th.conv}]
Consider $w_*\in\omega(v_0)$. Since $w_*$ is radially symmetric by
Definition~\ref{definOL}, we obtain from Proposition~\ref{prA18}
that, given any $\sigma\in \mathbb{S}^{N-1}$, the function
$W_{*,\sigma}: r\mapsto w_*(r\sigma)$ is a solution to the ordinary
differential equation \eqref{SSODE} which is non-negative in
$(0,\infty)$. Moreover,
$$
\int_0^\infty e^r\ W_{*,\sigma}(r)^2\ r^{N-1}\ dr \le C \int_{\mathbb{R}^N} e^{|x|} w_*(x)^2\ dx < \infty\ ,
$$
and we infer from Theorem~\ref{th.uniq} that either $W_{*,\sigma} = f(\cdot;a_*)$ or $W_{*,\sigma}=0$. However, the lower bound \eqref{db175} excludes the latter as it guarantees that $\|W_{*,\sigma}\|_\infty = \|w_*\|_\infty \ge 1/C_4>0$. Consequently, $W_{*,\sigma} = f(\cdot;a_*)$ for any $\sigma\in\mathbb{S}^{N-1}$ so that $w_*(x) = f(|x|;a_*)$ for all $x\in\real^N$. This completes the proof.
\end{proof}

%%%%%%%%%%%%%%%%%%%%
%%%%%%%%%%%%%%%%%%%%
\subsection{Proof of Lemma~\ref{lemA120}}\label{s3.4}
%%%%%%%%%%%%%%%%%%%%
%%%%%%%%%%%%%%%%%%%%

\begin{proof}[Proof of Lemma~\ref{lemA120}]
For $t\in [0,T_e)$, the positivity of $\mathcal{I}(u(t))$ and $\mathcal{J}(u(t))$ is a straightforward consequence of the definition of the extinction time $T_e$. Also, since $2/(p-1)>1$, it readily follows from Lemma~\ref{lemA11} that
$\mathcal{I}(u(t))<\infty$. We next infer from Lemma~\ref{lemA11} and \eqref{db2} that
\begin{equation}
|\nabla u(t,x)| \le C_0 \left( 1 + \|u_0\|_\infty^{(2-p)/p} t^{-1/p}
\right)\ u(t,x) \le C \left( 1 + t^{-1/p} \right) e^{-|x|/(p-1)}
\label{db7}
\end{equation}
for $x\in \mathbb{R}^N$, from which we deduce that
$\mathcal{J}(u(t))<\infty$ after noticing that $p/(p-1)>1$. It also follows from \cite{DBFr85a, DBFr85b} that both $u$ and $\nabla u$ are locally H\"older continuous in $(0,\infty)\times \mathbb{R}^N$. Together with the continuity of $u$ in $[0,\infty)\times\mathbb{R}^N$, Lemma~\ref{lemA11}, and
\eqref{db7}, this implies that
\begin{equation}
\mathcal{I}(u)\in C([0,T_e]) \;\;\mbox{ and }\;\; \mathcal{J}(u)\in
C((0,T_e])\ . \label{db7b}
\end{equation}

We next turn to the derivation of \eqref{z1}-\eqref{z2} and mention
that the computations given below are partly formal as they use more
regularity on $u$ than that available, in particular the local
square integrability of $\partial_t u$ in
$(0,T_e)\times\mathbb{R}^N$. A fully rigorous proof requires to
replace $u$ by classical solutions $(u_\varepsilon)_\varepsilon$ to
regularized non-degenerate parabolic equations which converge to $u$
as $\varepsilon\to 0$. The computations performed below are then
done with $u_\varepsilon$ and one first lets $\varepsilon\to 0$
before passing to the limit with the other parameters such as $R$,
see below. For simplicity, we omit this step here and refer to
\cite{IaLa12} for a complete description of the approximation
procedure.

Consider a radially symmetric cut-off function $\vartheta\in C_0^\infty(\mathbb{R}^N)$ satisfying $\vartheta(x)=1$ for $x\in
B(0,1)$, $\vartheta(x)=0$ for $x\not\in B(0,2)$, $x\cdot \nabla \vartheta(x)\le 0$ for $x\in\mathbb{R}^N$, and $\nabla\vartheta/\sqrt{\vartheta}\in L^\infty(\mathbb{R}^N)$. For
$R>1$ and $x\in\mathbb{R}^N$, we set $\vartheta_R(x) :=
\vartheta(x/R)$ and infer from \eqref{eq1} and \eqref{db3} that
\begin{align*}
& \frac{1}{2} \frac{d}{dt} \int_{\mathbb{R}^N} e^{|x|} \vartheta_R(x) |u(t,x)|^2\ dx \\
& \qquad = - \int_{\mathbb{R}^N} \nabla\left( e^{|x|} \vartheta_R(x) u(t,x) \right) \cdot |\nabla u(t,x)|^{p-2} \nabla u(t,x)\ dx \\
& \qquad\qquad - \int_{\mathbb{R}^N} e^{|x|} \vartheta_R(x) |\nabla u(t,x)|^{p-1} u(t,x)\ dx \\
& \qquad = - \int_{\mathbb{R}^N} e^{|x|} u(t,x) |\nabla u(t,x)|^{p-2} \nabla u(t,x)\cdot \nabla\vartheta_R(x)\ dx \\
& \qquad\qquad - \int_{\mathbb{R}^N} e^{|x|} \vartheta_R(x) \left[ |\nabla u(t,x)|^p + u(t,x) |\nabla u(t,x)|^{p-2} \nabla u(t,x)\cdot \frac{x}{|x|} \right]\ dx \\
& \qquad\qquad + \int_{\mathbb{R}^N} e^{|x|} \vartheta_R(x) u(t,x) |\nabla u(t,x)|^{p-2} \nabla u(t,x)\cdot \frac{x}{|x|}\ dx\\
& \qquad = - \int_{\mathbb{R}^N} e^{|x|} u(t,x) |\nabla u(t,x)|^{p-2} \nabla u(t,x)\cdot \nabla\vartheta_R(x)\ dx \\
& \qquad\qquad - \int_{\mathbb{R}^N} e^{|x|} \vartheta_R(x) |\nabla
u(t,x)|^p\ dx\ .
\end{align*}
Since
\begin{align*}
& \left| \int_{\mathbb{R}^N} e^{|x|} u(t,x) |\nabla u(t,x)|^{p-2} \nabla u(t,x)\cdot \nabla\vartheta_R(x)\ dx \right| \\
& \qquad \le \frac{C^{p-1} \kappa_0 \|\nabla\vartheta\|_\infty}{R}
\left( 1 + t^{-1/p} \right)^{p-1} \int_{\mathbb{R}^N}
e^{-|x|/(p-1)}\ dx
\end{align*}
by Lemma~\ref{lemA11} and \eqref{db7}, the first term of the
right-hand side of the previous inequality vanishes as $R\to\infty$
and we conclude that
\begin{equation}
\frac{d}{dt} \mathcal{I}(u(t)) = - p \mathcal{J}(u(t))\ , \qquad
t\in (0,T_e)\ . \label{db8}
\end{equation}
In particular, $\mathcal{I}(u)\in C^1((0,T_e])$ which completes the proof of \eqref{z0}.

It next follows from \eqref{eq1} and \eqref{db3} that
\begin{align*}
& \frac{1}{p} \frac{d}{dt} \int_{\mathbb{R}^N} e^{|x|} \vartheta_R(x) |\nabla u(t,x)|^p\ dx \\
& \qquad = - \int_{\mathbb{R}^N} \mathrm{div}\left( e^{|x|} \vartheta_R(x) |\nabla u(t,x)|^{p-2} \nabla u(t,x) \right) \partial_t u(t,x)\ dx \\
& \qquad = - \int_{\mathbb{R}^N} e^{|x|} |\nabla u(t,x)|^{p-2} \nabla u(t,x) \cdot \nabla\vartheta_R(x) \partial_t u(t,x)\ dx \\
& \qquad\qquad - \int_{\mathbb{R}^N} e^{|x|} \vartheta_R(x) \left[
|\nabla u(t,x)|^{p-2} \nabla u(t,x) \cdot \frac{x}{|x|} + \Delta_p
u(t,x) \right] \partial_t u(t,x)\ dx \ .
\end{align*}
Using again \eqref{eq1} and \eqref{db3} we end up with
\begin{align}
& \frac{1}{p} \frac{d}{dt} \int_{\mathbb{R}^N} e^{|x|} \vartheta_R(x) |\nabla u(t,x)|^p\ dx \nonumber\\
& \qquad = - \int_{\mathbb{R}^N} e^{|x|} \vartheta_R(x) |\partial_t u(t,x)|^2\ dx \nonumber\\
& \qquad\qquad - \int_{\mathbb{R}^N} e^{|x|} |\nabla u(t,x)|^{p-2} \nabla
u(t,x) \cdot \nabla\vartheta_R(x) \partial_t u(t,x)\ dx \ .
\label{db9}
\end{align}
For $\delta\in (0,1)$ we infer from the Cauchy-Schwarz inequality, the
boundedness of $\nabla\vartheta/\sqrt{\vartheta}$, and \eqref{db7}
that
\begin{align*}
& \left| \int_{\mathbb{R}^N} e^{|x|} |\nabla u(t,x)|^{p-2} \nabla u(t,x) \cdot \nabla\vartheta_R(x) \partial_t u(t,x)\ dx \right| \\
& \qquad \le \delta \int_{\mathbb{R}^N} e^{|x|} \vartheta_R(x) |\partial_t u(t,x)|^2\ dx + \frac{1}{4\delta} \int_{\mathbb{R}^N} e^{|x|} \left| \frac{\nabla\vartheta_R}{\sqrt{\vartheta_R}} \right|^2  |\nabla u(t,x)|^{2(p-1)}\ dx \\
& \qquad \le \delta \int_{\mathbb{R}^N} e^{|x|} \vartheta_R(x) |\partial_t u(t,x)|^2\ dx + \frac{C}{\delta R} \left( 1+ t^{-1/p} \right)^{2(p-1)} \int_{\mathbb{R}^N} e^{-|x|} \ dx \\
& \qquad \le \delta \int_{\mathbb{R}^N} e^{|x|} \vartheta_R(x)
|\partial_t u(t,x)|^2\ dx + \frac{C}{\delta R} \left( 1+
t^{-2(p-1)/p} \right)\ .
\end{align*}
Combining the above inequality with \eqref{db9} gives, after
integration over $(t_1,t_2)\subset (0,T_e)$,
\begin{align}
& \frac{1}{p} \int_{\mathbb{R}^N} e^{|x|} \vartheta_R(x) |\nabla u(t_2,x)|^p\ dx + (1-\delta) \int_{t_1}^{t_2} \int_{\mathbb{R}^N} e^{|x|} \vartheta_R(x) |\partial_t u(t,x)|^2\ dxdt \nonumber\\
& \qquad \le \mathcal{J}(u(t_1)) + \frac{C}{\delta R}
\int_{t_1}^{t_2} \left( 1+ t^{-2(p-1)/p} \right)\ dt\ . \label{db10}
\end{align}
As $p\in (1,2)$, one has $2(p-1)/p<1$ and we infer from \eqref{db10}
that
$$
(1-\delta) \int_0^{T_e} \int_{\mathbb{R}^N} e^{|x|} \vartheta_R(x)
|\partial_t u(t,x)|^2\ dxdt \le \mathcal{J}(u_0) + \frac{C}{R}
\left( T_e + T_e^{(2-p)/p} \right)\ .
$$
The Fatou lemma then entails that $\partial_t u$ belongs to
$L^2((0,T_e)\times\mathbb{R}^N; e^{|x|}\ dxdt)$ and we may let $R\to
\infty$ in \eqref{db10} to deduce that
$$
\mathcal{J}(u(t_2)) + (1-\delta) \int_{t_1}^{t_2}
\int_{\mathbb{R}^N} e^{|x|}  |\partial_t u(t,x)|^2\ dxdt \le
\mathcal{J}(u(t_1))\ .
$$
Letting $\delta\to 0$ finally gives \eqref{z2}.
\end{proof}

%%%%%%%%%%%%%%%%%%%%
%%%%%%%%%%%%%%%%%%%%
\section{Classification of self-similar profiles}\label{sec.s4}
%%%%%%%%%%%%%%%%%%%%
%%%%%%%%%%%%%%%%%%%%

This section is devoted to the proof of Theorem~\ref{th.uniq},
which, besides the uniqueness of the stabilization profile which is
at the heart of the proof of Theorem~\ref{th.conv}, also provides a
complete classification of the possible behaviors of solutions to
\eqref{SSODE}. A by-product of this section is the description of
all non-negative radially symmetric solutions to \eqref{eq1} of the
form \eqref{sepvar}.

Recall that the purpose of this section is to study the behavior of solutions to the initial boundary value problem
\begin{eqnarray}
&&\left(|f'|^{p-2}f'\right)'(r)+\frac{N-1}{r}\left(|f'|^{p-2}f'\right)(r)+f(r)-|f'(r)|^{p-1}=0, \qquad r>0,\label{ODE}
\\
&&f(0)=a, \ f'(0)=0, \label{IC}
\end{eqnarray}
according to the value of $a>0$. We first observe that, introducing $g:= -|f'|^{p-2} f'$, an alternative formulation of \eqref{ODE}-\eqref{IC} reads
\begin{align*}
f'(r) = - |g(r)|^{(2-p)/(p-1)} g(r)\ , & \qquad g'(r) + \frac{N-1}{r} g(r) + f(r) - |g(r)| = 0\ , \qquad r>0\ , \\
f(0) = a\ , & \qquad g(0)=0\ .
\end{align*}
Since $1/(p-1)>1$, standard results ensure that this initial value problem has a unique solution $(f,g)(\cdot;a)\in C^1([0,\mathcal{R}(a)))$ defined in a maximal interval of existence $[0,\mathcal{R}(a))$ for some $\mathcal{R}(a)\in(0,\infty]$. In addition, either $\mathcal{R}(a) = \infty$ or
\begin{equation}
\mathcal{R}(a)<\infty \;\;\text{ and }\;\; \limsup_{r\to\mathcal{R}(a)}\ (|f(r;a)|+|g(r;a)|) = \infty\ . \label{altBU}
\end{equation}
Clearly $f(\cdot;a)$ solves \eqref{ODE}-\eqref{IC} and, since $f(0;a)=a>0$, we may define the first zero of $f(\cdot;a)$ by
\begin{equation}\label{support}
R(a):=\inf\{r \in [0,\mathcal{R}(a))\ :\ f(r;a)=0\}\in (0,\mathcal{R}(a)]\ ,
\end{equation}
whenever it exists. When no confusion may arise we shall omit the dependence of $f(\cdot;a)$ on $a$ and simply use the notation $f$ instead without further notice.

Let us recall that we assume throughout all this section that \emph{the space dimension satisfies $N\geq 2$}. As already mentioned in the introduction, the proof of Theorem~\ref{th.uniq} for $N=1$ requires a different technique and is done in \cite{IL16}. The study of \eqref{ODE}-\eqref{IC} is divided into a number of lemmas that we
give in the sequel.

%%%%%%%%%%%%%%%%%%%%
%%%%%%%%%%%%%%%%%%%%
\subsection{Basic properties}\label{s4.1}
%%%%%%%%%%%%%%%%%%%%
%%%%%%%%%%%%%%%%%%%%

In the next lemma we gather some elementary properties of solutions $f(\cdot;a)$ to \eqref{ODE}-\eqref{IC} for $a>0$.

%%%%%%%%%%%%%%%%%%%%
\begin{lemma}\label{lem41}
Let $a>0$.
\begin{itemize}
\item [(a)] For $r\in(0,R(a))$ the following bounds hold true:
\begin{equation}\label{f1}
0<f(r;a)<a, \qquad -\left(\frac{a}{N}r\right)^{1/(p-1)}<f'(r;a)<0.
\end{equation}
\item [(b)] In the interval $(0,R(a))$, \eqref{ODE} can be written in the following
alternative form:
\begin{equation}\label{f2}
\frac{d}{dr}\left[\varrho(r)|f'(r;a)|^{p-2}f'(r;a)\right]=-\varrho(r)f(r;a),
\qquad r\in(0,R(a))\ ,
\end{equation}
where
\begin{equation}
\varrho(r)=r^{N-1}e^r\ , \qquad r\ge 0\ . \label{fctrho}
\end{equation}
\item [(c)] The maximal existence time is infinite, that is, $\mathcal{R}(a)=\infty$.
\item [(d)] If $R(a)=\infty$, then
$$
\lim\limits_{r\to\infty}f(r;a)=\lim\limits_{r\to\infty}f'(r;a)=0.
$$
\end{itemize}
\end{lemma}
%%%%%%%%%%%%%%%%%%%%

\begin{proof}
\textbf{Properties~(a)-(b).} It readily follows from \eqref{ODE} that, for $r>0$, 
\begin{equation}\label{interm1}
\frac{d}{dr}\left[\varrho(r)|f'(r)|^{p-2}f'(r)\right]=\varrho(r)\left[|f'(r)|^{p-2}f'(r)+|f'(r)|^{p-1}-f(r)\right].
\end{equation}
Since
$$
\lim\limits_{r\to0}(|f'|^{p-2}f')'(r)=-\frac{a}{N}<0, \qquad
f'(0)=0,
$$
we deduce that $f'$ is negative in a right neighborhood of $r=0$. As
long as this is true, we infer from \eqref{interm1} that
$$
\frac{d}{dr}\left[\varrho(r)|f'(r)|^{p-2}f'(r)\right]=-\varrho(r)f(r)<0,
$$
whence
$$
\varrho(r)|f'(r)|^{p-2}f'(r)+\int_0^r\varrho(\sigma)f(\sigma)\,d\sigma=0.
$$
Thus, $f'$ cannot vanish in $(0,R(a))$. Consequently, $f'<0$ in
$(0,R(a))$ and \eqref{f2} follows from \eqref{interm1}. Furthermore,
$f$ is decreasing on $(0,R(a))$, so that $0<f(r)<a$ for any $r\in(0,R(a))$. We then deduce from \eqref{f2} that
$$
\frac{d}{dr}\left[\varrho(r)|f'(r)|^{p-2}f'(r)\right]\geq-a\varrho(r)\ , \qquad r\in(0,R(a))\ .
$$
Integrating the above inequality gives, for $r\in (0,R(a))$,
$$
\varrho(r)|f'(r)|^{p-2}f'(r)\geq-a\int_0^r\varrho(\sigma)\,d\sigma\geq-ae^r\int_0^r\sigma^{N-1}\,d\sigma,
$$
whence $|f'(r)|^{p-2}f'(r)\geq-ar/N$, leading to the second bound in \eqref{f1}.

\medskip

\noindent\textbf{Property~(c).} Define the following energy
\begin{equation}\label{energy}
E(r):=\frac{p-1}{p}|f'(r)|^{p}+\frac{1}{2}f(r)^2, \qquad
r\in(0,\mathcal{R}(a)).
\end{equation}
We readily infer from \eqref{ODE} and the definition of the energy that, for $r\in [0,\mathcal{R}(a))$,
\begin{equation*}
\begin{split}
E'(r)&=f'(r)\left[(|f'|^{p-2}f')'(r)+f(r)\right]=f'(r)|f'(r)|^{p-1}-\frac{N-1}{r}|f'(r)|^p\\
&\leq|f'(r)|^p
=\frac{p}{p-1}\left[E(r)-\frac{1}{2}f(r)^2\right]\leq\frac{p}{p-1}E(r)\ .
\end{split}
\end{equation*}
Thus $E$ does not blow up in finite time which excludes the alternative \eqref{altBU} and entails that $\mathcal{R}(a)=\infty$.

\medskip

\noindent\textbf{Property~(d).} Assume now that $R(a)=\infty$. Recalling the energy
introduced in \eqref{energy}, we infer from \eqref{f1}  that
\begin{equation}\label{interm2}
%\begin{split}
E'(r)=f'(r)\left[|f'(r)|^{p-1}-\frac{N-1}{r}(|f'|^{p-2}f')(r)\right]=-\left(1+\frac{N-1}{r}\right)|f'(r)|^p\leq 0
%\end{split}
\end{equation}
for $r\in (0,\infty)$. Since $E$ is non-negative in $(0,\infty)$, it follows that there exist $l_f\ge 0$ and $l_E\geq 0$ such that
$$
\lim\limits_{r\to\infty}f(r)=l_f, \quad
\lim\limits_{r\to\infty}E(r)=l_E.
$$
From the definition of $E$ we deduce that $-f'(r)=|f'(r)|$ also has a limit as $r\to\infty$, while \eqref{interm2} and the existence of
$l_E\geq 0$ entail that $f'\in L^p(0,\infty)$. Therefore,
\begin{equation}
\label{interm3}\lim\limits_{r\to\infty}f'(r)=0.
\end{equation}
Assume now for contradiction that $l_f>0$. We infer from \eqref{ODE}
and \eqref{interm3} that
$$
\lim\limits_{r\to\infty}(|f'|^{p-2}f')'(r)=-l_f,
$$
so that there exists $R>0$ such that
$$
(|f'|^{p-2}f')'(r)\leq-\frac{l_f}{2} \qquad {\rm for \ any} \ r\geq
R.
$$
After integration, we get
$$
(|f'|^{p-2}f')(r)\leq(|f'|^{p-2}f')(R)-\frac{l_f}{2}(r-R), \qquad
r\geq R,
$$
which implies that $|f'(r)|^{p-1}=-(|f'|^{p-2}f')(r)$ diverges to
$\infty$ as $r\to\infty$ and contradicts
\eqref{interm3}. Therefore $l_f=0$, ending the proof.
\end{proof}

The following lemma shows that $R(a)$ is indeed finite for some values of $a$.

%%%%%%%%%%%%%%%%%%%%
\begin{lemma}\label{lem42}
There exists $a_{\infty}>0$ such that $R(a)\in(0,\infty)$ for any
$a>a_{\infty}$.
\end{lemma}
%%%%%%%%%%%%%%%%%%%%

\begin{proof}
For $a>0$, introduce the following rescaling
$$
f(r;a)=a\varphi\left(a^{(2-p)/p}r;a\right), \quad
f'(r;a)=a^{2/p}\varphi'\left(a^{(2-p)/p}r;a\right), \quad
r\in[0,R(a)).
$$
Letting $s:=a^{(2-p)/p}r$ and dropping the parameter $a$ from the
notation, it follows from \eqref{ODE} that $\varphi$ solves the
following differential equation
\begin{equation}\label{interm4}
(|\varphi'|^{p-2}\varphi')'(s)+\frac{N-1}{s}(|\varphi'|^{p-2}\varphi')(s)+\varphi(s)-a^{(p-2)/p}|\varphi'(s)|^{p-1}=0,
\end{equation}
for $s\in(0,\infty)$, with initial conditions
$\varphi(0)=1$, $\varphi'(0)=0$. Let us also consider the solution
$\psi$ to the following problem
\begin{equation}\label{interm5}
\left\{\begin{array}{l}
\displaystyle{(|\psi'|^{p-2}\psi')'(s)+\frac{N-1}{s}(|\psi'|^{p-2}\psi')(s)+\psi(s)=0,
\quad s>0,} \\
\\
\psi(0)=1, \ \psi'(0)=0.
\end{array}\right.
\end{equation}
Note that, since $p<2$, continuous dependence guarantees that
$$
\lim\limits_{a\to\infty}\sup\limits_{[0,s]}|\varphi(\cdot;a)-\psi|=0 \qquad {\rm for \ all} \ s>0.
$$
We claim that there is $s_0\in(0,\infty)$ such that $\psi(s_0)=0$,
$\psi'(s_0)<0$, and $\psi'(s)<0<\psi(s)$ for $s\in(0,s_0)$. The fact
that there exists a maximal interval $(0,s_0)$ on which $\psi>0$ and
$\psi'<0$ is easily proved as in Lemma \ref{lem41}, so it remains to
check that $s_0<\infty$. Assume for contradiction that $s_0=\infty$,
so that $\psi$ is positive and decreasing in $(0,\infty)$. In the
analysis below, we improve on an idea coming from \cite{Shi04}. Fix
$\vartheta\in((N-p)/N,p-1)$, which is possible since
$$
p-1-\frac{N-p}{N}=\frac{p(N+1)-2N}{N}=\frac{N+1}{N}(p-p_c)>0,
$$
and introduce
$$
\delta:=\left[\int_0^1\sigma^{N-1}\psi(\sigma)\,d\sigma\right]^{(1-\vartheta)/(p-1)}>0.
$$
It follows from the monotonicity and positivity of $\psi$ that, for
any $s\geq1$,
\begin{equation}\label{interm6}
\begin{split}
\int_0^s\sigma^{N-1}\psi(\sigma)\,d\sigma&\geq\left[\int_0^1\sigma^{N-1}\psi(\sigma)\,d\sigma\right]^{1-\vartheta}
\left[\int_0^s\sigma^{N-1}\psi(\sigma)\,d\sigma\right]^{\vartheta}\\
&\geq\delta^{p-1}\frac{s^{N\vartheta}}{N^{\vartheta}}\psi(s)^{\vartheta}.
\end{split}
\end{equation}
We then infer from \eqref{interm5} that, for any $s\geq1$,
$$
\frac{d}{ds}\left[s^{N-1}(|\psi'|^{p-2}\psi')(s)\right]=-s^{N-1}\psi(s),
$$
hence, taking into account \eqref{interm6},
$$
(|\psi'|^{p-2}\psi')(s)=-\frac{1}{s^{N-1}}\int_0^s\sigma^{N-1}\psi(\sigma)\,d\sigma\leq-\frac{\delta^{p-1}}{N^{\vartheta}}s^{N(\vartheta-1)+1}\psi(s)^{\vartheta},
$$
or equivalently, since $\psi>0$ and $\psi'<0$ in $(0,\infty)$,
$$
-\psi(s)^{-\vartheta/(p-1)}\psi'(s)\geq\frac{\delta}{N^{\vartheta/(p-1)}}s^{(N(\vartheta-1)+1)/(p-1)}.
$$
By integration on the interval $[1,s]$ and straightforward
manipulations, we thus obtain that
$$
\psi(s)^{(p-1-\vartheta)/(p-1)}\leq\psi(1)^{(p-1-\vartheta)/(p-1)}+\frac{\delta(p-1-\vartheta)}{(p-N+N\vartheta)N^{\vartheta/(p-1)}}
\left[1-s^{(p+N(\vartheta-1))/(p-1)}\right].
$$
We deduce from the definition of $\vartheta$ that $p-1-\vartheta>0$
and $p+N(\vartheta-1)>0$, so that $\psi(s)^{(p-1-\vartheta)/(p-1)}$
becomes negative for $s$ sufficiently large, hence a contradiction.
Consequently, $s_0<\infty$ and
$$
s_0^{N-1}|\psi'(s_0)|^{p-2}\psi'(s_0)=-\int_0^{s_0}\sigma^{N-1}\psi(\sigma)\,d\sigma<0,
$$
so that $\psi'(s_0)<0$. Thus, for $\e>0$ sufficiently small, there
are $s_1<s_0<s_2$ such that $\psi(s_1)>\e>-\e>\psi(s_2)$. By
continuous dependence we realize that
$$
\varphi(s_1;a)>\frac{\e}{2}>-\frac{\e}{2}>\varphi(s_2;a),
$$
for $a$ sufficiently large. Therefore $\varphi(\cdot;a)$ vanishes in $(0,\infty)$ when $a$ is sufficiently large and so does $f(\cdot;a)$, which completes the proof.
\end{proof}

%%%%%%%%%%%%%%%%%%%%
%%%%%%%%%%%%%%%%%%%%
\subsection{Classification of behaviors as $r\to\infty$}\label{s4.2}
%%%%%%%%%%%%%%%%%%%%
%%%%%%%%%%%%%%%%%%%%

In this subsection we analyze the possible decay rates as $r\to\infty$ of solutions to \eqref{ODE}-\eqref{IC}. To this end, let $a>0$ and recall that the function $g(\cdot;a)$ is defined by
$$
g(r;a)=-|f'(r;a)|^{p-2}f'(r;a)\ , \qquad r\ge 0\ .
$$
Since $f'(\cdot;a)<0$ on $(0,R(a))$ by Lemma~\ref{lem41}~(a), it
follows that $g(\cdot;a)=|f'(\cdot;a)|^{p-1}>0$ on $(0,R(a))$ and
the first positive zero of $g(\cdot;a)$ (if any) thus satisfies
\begin{equation}\label{r1a}
R_1(a):=\sup\{R>0: g(\cdot;a)>0 \ {\rm on} \ (0,R)\} \ge R(a)\ ,
\end{equation}
where $R_1(a)=\infty$ if $g(\cdot;a)>0$ on $(0,\infty)$. From \eqref{ODE} we readily infer that
$$
-g'(r)-\left(1+\frac{N-1}{r}\right)g(r)+f(r)=0, \qquad r\in
(0,R_1(a)),
$$
so that, differentiating with respect to $r$, we get
\begin{equation}\label{ODEg}
g''(r)+\left(1+\frac{N-1}{r}\right)g'(r)+g(r)^{1/(p-1)}-\frac{N-1}{r^2}g(r)=0,
\qquad r\in(0,R_1(a)),
\end{equation}
where for simplicity we denoted $g=g(\cdot;a)$. Furthermore,
\begin{equation}\label{ICg}
g(0)=0, \quad g'(0)=\frac{a}{N}>0.
\end{equation}
We also introduce the function
\begin{equation}\label{f3}
w(r)=w(r;a):=\varrho(r)g(r;a), \qquad r\in [0,R_1(a)),
\end{equation}
and by straightforward calculations we obtain that $w$ solves
\begin{align}
& w''(r)-\left(1+\frac{N-1}{r}\right)w'(r)+\varrho(r)^{-(2-p)/(p-1)}w(r)^{1/(p-1)}=0, \quad r\in (0,R_1(a)),
\label{ODEw} \\ & w(0)=0, \quad w'(r)\sim a\varrho(r) \ {\rm as} \
r\to 0.\label{ICw}
\end{align}

%%%%%%%%%%%%%%%%%%%%
\begin{lemma}\label{lem43}
The function $g(\cdot;a)$ is positive on $(0,\infty)$ (or, equivalently, $R_1(a)=\infty$) if and
only if $w'(\cdot;a)>0$ on $(0,\infty)$.
\end{lemma}
%%%%%%%%%%%%%%%%%%%%

\begin{proof}
We derive from \eqref{ICg} and \eqref{f3} that
$$
w'(r)=\varrho(r)\left[\left(1 + \frac{N-1}{r} \right)g(r)+g'(r)\right]\sim\varrho(r)\left[ \left( 1 + \frac{N-1}{r} \right) \frac{ar}{N} + \frac{a}{N} \right]\sim
a\varrho(r)
$$
as $r\to 0$. Therefore, $w'>0$ in a right neighborhood of $r=0$.
Define then
$$
r_0:=\sup\{R>0: w'>0 \ {\rm in} \ (0,R)\} > 0.
$$
Assume first that $R_1(a)=\infty$ and assume for contradiction that
$r_0<\infty$. Then $w''(r_0)<0$ by \eqref{ODEw}, thus $w'$ is
negative in a right neighborhood of $r_0$. In addition,
\eqref{ODEw} guarantees that $w''(r)<0$ for any $r>0$ such that $w'(r)<0$, a property which readily implies that $w'(r)<0$ and $w''(r)<0$ for $r\in(r_0,\infty)$. Now fix $r_1>r_0$. Then $w'(r)\leq w'(r_1)<0$ for $r\geq r_1$, hence
$$
w(r)\leq  w(r_1)+w'(r_1)(r-r_1), \qquad r\geq r_1,
$$
which implies that $w(r)$ vanishes at a finite $r>r_1$,
contradicting the assumption $R_1(a)=\infty$. Consequently,
$r_0=\infty$ and $w'>0$ in $(0,\infty)$.

The converse assertion is obvious: if $w'>0$ in $(0,\infty)$, then $w(r)>w(0)=0$ for $r\in (0,\infty)$ and $R_1(a)=\infty$.
\end{proof}

We now split the range of $a$ into the following sets:
\begin{eqnarray*}
&&A:=\{a\in(0,\infty): R_1(a)<\infty\}, \\
&&B:=\{a\in(0,\infty): R_1(a)=\infty \ {\rm and} \ w(\cdot;a) \ {\rm
is \
bounded}\},\\
&&C:=\{a\in(0,\infty): R_1(a)=\infty \ {\rm and} \ w(\cdot;a) \ {\rm
is \ unbounded}\}
\end{eqnarray*}
From Lemma \ref{lem43} we infer that $A\cup B\cup C=(0,\infty)$ and, owing to the monotonicity and positivity of $w(\cdot;a)$ for $a\in B\cup C$,
\begin{equation}
\lim\limits_{r\to\infty}w(r;a)=\left\{
\begin{array}{ll}
l(a)\in(0,\infty) & {\rm if} \ a\in B\ , \\
& \\
\infty & {\rm if} \ a\in C\ .
\end{array}
\right. \label{limitw}
\end{equation}
The following result relates $R_1(a)$ and $R(a)$.

%%%%%%%%%%%%%%%%%%%%
\begin{lemma}\label{lem44}
We have $R_1(a)<\infty$ if and only if $R(a)<\infty$.
\end{lemma}
%%%%%%%%%%%%%%%%%%%%

\begin{proof}
Recall first that $R(a)\leq R_1(a)$ by \eqref{r1a}, so that the finiteness of $R_1(a)$ implies that of $R(a)$. Assume next that $R(a)<\infty$. Then $f'(R(a))<0$, which implies that $g(R(a))>0$ and thus
$R_1(a)>R(a)$. Furthermore, $g>0$ in $(R(a),R_1(a))$ which entails that $f(r)<f(R(a))=0$ for $r\in (R(a),R_1(a))$. Fix $\theta\in(R(a),R_1(a))$. Since
$f(\theta)<0$ and $f'<0$ in $(R(a),R_1(a))$, we obtain from \eqref{f2} and the definition of the function $g$ that
$$
\frac{d}{dr}(\varrho(r)g(r))=\varrho(r)f(r)\leq\varrho(r)f(\theta),
\qquad r\in[\theta,R_1(a)),
$$
whence, by integration,
$$
\varrho(r)g(r)\leq\varrho(\theta)g(\theta)+ f(\theta)\, \int_{\theta}^r\varrho(s) \,ds\ , \qquad r\in (\theta,R_1(a))\ .
$$
Since $f(\theta)<0$ and $\varrho(r)\to\infty$ as $r\to\infty$, the right-hand side of the above inequality is negative for $r$ large enough, which excludes that $R_1(a)=\infty$. Therefore $R_1(a)<\infty$ and the proof is complete.
\end{proof}

%%%%%%%%%%%%%%%%%%%%
\begin{corollary}\label{corA}
The set $A$ is non-empty and open.
\end{corollary}
%%%%%%%%%%%%%%%%%%%%

\begin{proof}
The fact that $A$ is open follows by continuous dependence with
respect to the parameter $a$. Moreover, we infer from Lemma
\ref{lem42} and Lemma \ref{lem44} that there exists $a_{\infty}>0$
such that $(a_{\infty},\infty)\subseteq A$, hence $A$ is non-empty.
\end{proof}

%%%%%%%%%%%%%%%%%%%%
\begin{lemma}\label{lem45}
Let $a\in B\cup C$. Then
\begin{eqnarray}
&&\frac{w(\cdot;a)}{\varrho}\in L^{1/(p-1)}(0,\infty), \qquad
\lim\limits_{r\to\infty}\frac{w(r;a)}{\varrho(r)}=0, \label{f4} \\
&&\lim\limits_{r\to\infty}\frac{w'(r;a)}{\varrho(r)}=0.\label{f5}
\end{eqnarray}
\end{lemma}
%%%%%%%%%%%%%%%%%%%%

\begin{proof}
It readily follows from \eqref{ODEw} that
\begin{equation}\label{interm7}
\frac{d}{dr}\left(\frac{w'(r)}{\varrho(r)}\right)=-\left(\frac{1}{\varrho(r)}\right)^{1+(2-p)/(p-1)}w(r)^{1/(p-1)}=-\left(\frac{w(r)}{\varrho(r)}\right)^{1/(p-1)}.
\end{equation}
Since $w'>0$ in $(0,\infty)$ for $a\in B\cup C$ by Lemma
\ref{lem43}, we infer from \eqref{interm7} that $w'(r)/\varrho(r)$
is non-increasing and non-negative. Consequently, there exists
$L\geq0$ such that
\begin{equation}\label{interm8}
\lim\limits_{r\to\infty}\frac{w'(r)}{\varrho(r)}=L.
\end{equation}
Since $w'(r)/\varrho(r)\to a$ as $r\to0$ by \eqref{ICw}, it readily follows from \eqref{interm7} and \eqref{interm8} that
\begin{equation}\label{interm9}
a-L=\int_0^{\infty}\left(\frac{w(r)}{\varrho(r)}\right)^{1/(p-1)}\,dr,
\end{equation}
which in particular gives that $w/\varrho\in L^{1/(p-1)}(0,\infty)$. Furthermore
\begin{equation}
\varrho'(r)=\left(1+\frac{N-1}{r}\right)\varrho(r)\sim\varrho(r) \quad {\rm as} \ r\to\infty, \label{asymptrho}
\end{equation}
and we deduce from \eqref{interm8} that $w'(r)/\varrho'(r)\to L$ as $r\to\infty$. Applying l'Hospital rule then gives
$$
\lim\limits_{r\to\infty}\frac{w(r)}{\varrho(r)}=L\ ,
$$
and the integrability of $(w/\varrho)^{1/(p-1)}$ implies $L=0$, which completes the proof of \eqref{f4} and \eqref{f5}.
\end{proof}

The next result goes deeper into the characterization of elements in the sets $B$ and $C$.

%%%%%%%%%%%%%%%%%%%%
\begin{lemma}\label{lem46}
If $a\in B\cup C$, then $w'(\cdot;a)/w(\cdot;a)$ has a limit as $r\to\infty$ and
\begin{equation}\label{f6}
\lim\limits_{r\to\infty}\frac{w'(r;a)}{w(r;a)}\in\{0,1\}.
\end{equation}
\end{lemma}
%%%%%%%%%%%%%%%%%%%%

\begin{proof}
Introducing $h:=w'/w$, it follows from \eqref{ODEw} and
easy algebraic manipulations that
\begin{equation}\label{ODEh}
h'(r)=\left(1+\frac{N-1}{r}\right)h(r)-h(r)^2-\left(\frac{w(r)}{\varrho(r)}\right)^{(2-p)/(p-1)},
\quad r>0.
\end{equation}
We infer from \eqref{ODEh} that
$$
\frac{d}{dr}\left(\frac{h(r)}{\varrho(r)}\right)=\frac{1}{\varrho(r)}\left[h'(r)-\left(1+\frac{N-1}{r}\right)h(r)\right]\leq-\frac{1}{\varrho(r)}h(r)^2,
$$
which can be written alternatively as
$$
\frac{d}{dr}\left(\frac{h(r)}{\varrho(r)}\right)+\varrho(r)\left(\frac{h(r)}{\varrho(r)}\right)^2\leq0
$$
or equivalently
$$
\varrho(r)\leq\frac{d}{dr}\left(\frac{\varrho(r)}{h(r)}\right).
$$
Noticing that
$$
\frac{\varrho(r)}{h(r)}=\frac{\varrho(r)}{w'(r)}w(r)\sim\frac{w(r)}{a}\to 0 \quad {\rm as} \ r\to 0
$$
by \eqref{ICw}, we may integrate the previous differential inequality and find
\begin{equation}\label{interm11}
h(r)\leq\varrho(r)\left(\int_0^r\varrho(s)\,ds\right)^{-1}, \quad
r\in(0,\infty).
\end{equation}
Since $\varrho'(r)/\varrho(r)\to1$ as $r\to\infty$ by \eqref{asymptrho}, the l'Hospital
rule applies and shows that
$$
\lim\limits_{r\to\infty}\varrho(r)\left(\int_0^r\varrho(s)\,ds\right)^{-1}=\lim\limits_{r\to\infty}\frac{\varrho'(r)}{\varrho(r)}=1,
$$
and we conclude from \eqref{interm11} that
\begin{equation}\label{interm12}
\limsup\limits_{r\to\infty}h(r)\leq1.
\end{equation}
Fix now $\e\in(0,1)$. According to \eqref{f4} and \eqref{interm12},
there exists $r_{\e}>0$ such that
\begin{equation}\label{interm13}
\left(\frac{w(r)}{\varrho(r)}\right)^{(2-p)/(p-1)}\leq\frac{\e^2}{4},
\quad r\geq r_{\e},
\end{equation}
and
\begin{equation}\label{interm14}
h(r)<r_2(\e):=\frac{1+\sqrt{1-\e^2}}{2}, \quad r\geq r_{\e}.
\end{equation}
Assume further that there exists $r_*>r_{\e}$ such that
\begin{equation}\label{interm15}
h(r_*) > r_1(\e):=\frac{1-\sqrt{1-\e^2}}{2},
\end{equation}
and define
$$
R_*:=\inf\{r>r_*: h(r)<r_1(\e)\}>r_*.
$$
For $r\in[r_*,R_*)$ we deduce from \eqref{ODEh}, \eqref{interm13},
\eqref{interm14}, and the definition of $R_*$ that
\begin{equation}\label{interm16}
h'(r)\geq h(r)-h(r)^2-\frac{\e^2}{4}=(r_2(\e)-h(r))(h(r)-r_1(\e))>0.
\end{equation}
It thus follows from \eqref{interm16} that $h$ is increasing in
$[r_*,R_*)$, which readily implies that $R_*=\infty$. Moreover, we
also get from \eqref{interm16} that
\begin{equation*}
\begin{split}
\frac{d}{dr}\log\left(\frac{h(r)-r_1(\e)}{r_2(\e)-h(r)}\right)&=\left(\frac{1}{h(r)-r_1(\e)}+\frac{1}{r_2(\e)-h(r)}\right)h'(r)\\
&=\frac{r_2(\e)-r_1(\e)}{(h(r)-r_1(\e))(r_2(\e)-h(r))}h'(r)\geq\sqrt{1-\e^2},
\end{split}
\end{equation*}
whence, by integration over $(r_*,r)$,
$$
\frac{h(r)-r_1(\e)}{r_2(\e)-h(r)}\geq K_{\e}e^{r\sqrt{1-\e^2}},
\quad
K_{\e}:=\frac{h(r_*)-r_1(\e)}{r_2(\e)-h(r_*)}e^{-r_*\sqrt{1-\e^2}}>0,
$$
the latter being true due to \eqref{interm14} and \eqref{interm15}.
Therefore,
\begin{equation*}\begin{split}
h(r)&\geq\frac{r_2(\e)K_{\e}e^{r\sqrt{1-\e^2}}+r_2(\e)+r_1(\e)-r_2(\e)}{1+K_{\e}e^{r\sqrt{1-\e^2}}}\\
&\geq r_2(\e)-\frac{\sqrt{1-\e^2}}{1+K_{\e}e^{r\sqrt{1-\e^2}}},
\end{split}\end{equation*}
for $r\geq r_*$. In particular,
\begin{equation}\label{interm17}
\liminf\limits_{r\to\infty}h(r)\geq
r_2(\e)=\frac{1+\sqrt{1-\e^2}}{2}\ ,
\end{equation}
this property being valid only if \eqref{interm15} holds true.

Suppose now that $h$ does not converge to zero as $r\to\infty$. Then
there exist $\mu\in(0,1/2)$ and a sequence $(r_j)_{j\ge 1}$ such that
$r_j\to\infty$ as $j\to\infty$ and
\begin{equation}\label{interm18}
h(r_j)\geq\mu, \quad j\geq1.
\end{equation}
Pick now $\e\in \left( 0,\sqrt{1-(1-2\mu)^2} \right)$. Then $(1-2\mu)<\sqrt{1-\e^2}$,
so that $r_1(\e)<\mu$. Also there exists $j_{\e}\geq1$ such that
$r_{j_{\e}}>r_{\e}$. Owing to \eqref{interm18} and the choice of
$\e$, it follows that $h$ satisfies \eqref{interm15} with
$r_*=r_{j_{\e}}$ and we deduce from the previous analysis that
\eqref{interm17} holds true. Since $\e>0$ can be picked as small as
we want in $\left( 0,\sqrt{1-(1-2\mu)^2} \right)$, we conclude that
$$
\liminf\limits_{r\to\infty}h(r)\geq1,
$$
which, together with \eqref{interm12}, leads to $h(r)\to1$ as
$r\to\infty$, completing the proof.
\end{proof}

As up to now all the previous steps were common to $a\in B$ and
$a\in C$, the following results introduce differences between the
sets $B$ and $C$.

%%%%%%%%%%%%%%%%%%%%
\begin{lemma}\label{lem47}
Let $a\in B\cup C$. Then
$$
a\in B \ {\rm if \ and \ only \ if} \
\lim\limits_{r\to\infty}\frac{w'(r;a)}{w(r;a)}=0.
$$
Furthermore, for $a\in B$,
\begin{equation}\label{f7}
w'(r;a)\sim (p-1)l(a)^{1/(p-1)}\varrho(r)^{-(2-p)/(p-1)} \quad {\rm
as} \ r\to\infty,
\end{equation}
where $l(a)$ is defined in \eqref{limitw}.
\end{lemma}
%%%%%%%%%%%%%%%%%%%%

\begin{proof}
Assume first that $a\in B$. Then $l(a)\in (0,\infty)$ and it follows from \eqref{interm7} that
$$
\lim\limits_{r\to\infty}\varrho(r)^{1/(p-1)}\frac{d}{dr}\left(\frac{w'}{\varrho}\right)(r)=
-\lim\limits_{r\to\infty} w(r)^{1/(p-1)} =-l(a)^{1/(p-1)},
$$
hence we deduce by the l'Hospital rule that
\begin{equation}\label{interm19}
\lim\limits_{r\to\infty}\frac{w'(r)}{\varrho(r)} \left( \int_r^{\infty}\varrho(s)^{-1/(p-1)}\,ds \right)^{-1}=l(a)^{1/(p-1)}.
\end{equation}
We next observe that
\begin{equation*}
\begin{split}
\int_{r}^{\infty}\varrho(s)^{-1/(p-1)}\,ds&=\int_{r}^{\infty}s^{-(N-1)/(p-1)}e^{-s/(p-1)}\,ds\\
&=(p-1)^{(p-N)/(p-1)}\int_{r/(p-1)}^{\infty}\sigma^{(p-N)/(p-1)-1}e^{-\sigma}\,d\sigma\\&=(p-1)^{(p-N)/(p-1)}\ \Gamma\left(\frac{p-N}{p-1},\frac{r}{p-1}\right),
\end{split}
\end{equation*}
where $\Gamma(\cdot,\cdot)$ is the (upper) incomplete Gamma function defined by
$$
\Gamma(\sigma,y):=\int_{y}^{\infty} z^{\sigma-1}e^{-z}\,dz\ , \qquad (\sigma,y)\in \real\times (0,\infty)\ .
$$
Since
$$
\Gamma(\sigma,y)\sim y^{\sigma-1}e^{-y} \ {\rm as} \ y\to\infty\ ,
$$
we obtain that, as $r\to\infty$,
$$
\int_r^{\infty}\varrho(s)^{-1/(p-1)}\,ds\sim(p-1)^{(p-N)/(p-1)}\left(\frac{r}{p-1}\right)^{-(N-1)/(p-1)}e^{-r/(p-1)}
$$
and thus
\begin{equation}\label{interm20}
\varrho(r)\int_r^{\infty}\varrho(s)^{-1/(p-1)}\,ds\sim(p-1)\varrho(r)^{-(2-p)/(p-1)}.
\end{equation}
We then derive easily \eqref{f7} from \eqref{interm19} and
\eqref{interm20}. Moreover $w'(r)\to 0$ as $r\to\infty$ by \eqref{f7} and thus $\lim\limits_{r\to\infty}w'(r)/w(r)=0$.

Conversely, assume that $\lim\limits_{r\to\infty}w'(r)/w(r)=0$.
Given $\delta\in(0,2-p)$, there exists $r_{\delta}>0$ such that
$0<w'(r)/w(r)\leq\delta$ for $r\geq r_{\delta}$, whence
\begin{equation}\label{interm21}
0 < w(r)\leq w(r_{\delta}) e^{\delta(r-r_{\delta})}\ , \quad 0< w'(r)\leq \delta w(r_{\delta}) e^{\delta(r-r_{\delta})}\ , \qquad r\geq r_{\delta}.
\end{equation}
In particular, $w'(r)/\varrho(r)\to 0$ as $r\to\infty$ and it follows from \eqref{interm7} that
$$
w'(r) = \varrho(r) \int_r^\infty \left(\frac{w(s)}{\varrho(s)}\right)^{1/(p-1)}\, ds\ , \qquad r>r_\delta\ .
$$
Combining the above inequality with \eqref{interm21} gives, for $r>r_\delta$,
\begin{align*}
w'(r) & \le w(r_\delta)^{1/(p-1)} e^{-\delta r_\delta/(p-1)} \varrho(r) \int_r^\infty s^{-(N-1)/(p-1)} e^{(\delta-1)s/(p-1)} \, ds \\
& \le C(\delta) \varrho(r) e^{(\delta-1)r/(p-1)} \frac{p-1}{N-p} r^{(p-N)/(p-1)} \\
& \le C(\delta) r^{((N+1)p-2N)/(p-1)} e^{(\delta+p-2)r/(p-1)}\ .
\end{align*}
We deduce from the choice of $\delta$, the above inequality, and \eqref{interm21} that $w'$ belongs to $L^1(r_\delta,\infty)$. Therefore, $w$ has a finite limit as $r\to\infty$ and thus $a\in B$.
\end{proof}

The previous lemma also allows us to identify the asymptotic behavior of $w(\cdot;a)$ for
$a\in C$.

%%%%%%%%%%%%%%%%%%%%
\begin{corollary}\label{cor.1}
Let $a\in B\cup C$. Then
$$
a\in C \ {\rm if \ and \ only \ if} \
\lim\limits_{r\to\infty}\frac{w'(r;a)}{w(r;a)}=1.
$$
Moreover, for $a\in C$ we have
\begin{equation}\label{f8}
w(r;a)\sim\varrho(r)\left(\frac{2-p}{p-1}r\right)^{-(p-1)/(2-p)}
\quad {\rm as} \ r\to\infty.
\end{equation}
\end{corollary}
%%%%%%%%%%%%%%%%%%%%

\begin{proof}
The first assertion follows readily from Lemma~\ref{lem46} and Lemma~\ref{lem47}. Since $w(r)\sim w'(r)$ as $r\to\infty$ for $a\in C$, we
derive from \eqref{interm7} that
\begin{equation*}\begin{split}
\frac{d}{dr}\left[\left(\frac{w'(r)}{\varrho(r)}\right)^{-(2-p)/(p-1)}\right]&=-\frac{2-p}{p-1}\left(\frac{w'(r)}{\varrho(r)}\right)^{-1/(p-1)}\frac{d}{dr}\left(\frac{w'(r)}{\varrho(r)}\right)\\
&=\frac{2-p}{p-1}\left(\frac{w(r)}{w'(r)}\right)^{-1/(p-1)},
\end{split}\end{equation*}
and thus, as $r\to\infty$,
$$
\left(\frac{w'(r)}{\varrho(r)}\right)^{-(2-p)/(p-1)}\sim\frac{2-p}{p-1}r,
$$
which easily implies \eqref{f8}.
\end{proof}

%%%%%%%%%%%%%%%%%%%%
\begin{corollary}\label{cor.2}
Let $a\in B\cup C$. Recalling that $g(\cdot;a) = - (|f'|^{p-2} f')(\cdot;a)$, there holds:
\begin{eqnarray*}
&&\lim\limits_{r\to\infty}\frac{g(r;a)}{f(r;a)}=\infty \quad {\rm if} \
a\in B,\\
&&\lim\limits_{r\to\infty}\frac{g(r;a)}{f(r;a)}=1 \quad {\rm if} \ a\in
C.
\end{eqnarray*}
\end{corollary}
%%%%%%%%%%%%%%%%%%%%

\begin{proof}
Observe that
$$
g=\frac{w}{\varrho}, \quad
g'=\frac{w'}{\varrho} - \frac{\varrho'}{\varrho}\, \frac{w}{\varrho},
\quad f=\frac{w'}{\varrho}.
$$
Consequently, for $a\in B\cup C$, there holds $g/f=w/w'$ on $(0,\infty)$, and the conclusion
readily follows from Lemma~\ref{lem47} and Corollary~\ref{cor.1}.
\end{proof}

%%%%%%%%%%%%%%%%%%%%
%%%%%%%%%%%%%%%%%%%%
\subsection{A Pohozaev functional}\label{s4.3}
%%%%%%%%%%%%%%%%%%%%
%%%%%%%%%%%%%%%%%%%%

The next step towards the identification of the sets $B$ and $C$ is
the construction of a \emph{Pohozaev functional}. Up to our
knowledge, the idea of considering such functionals to study the uniqueness of
solutions to some elliptic equation stems from Yanagida \cite{Ya91a,
Ya91b}, but the approach we use here is rather inspired by
\cite{ShWa13, ShWa16}.

Let $a>0$. Recalling that $g(\cdot;a) = - (|f'|^{p-2} f')(\cdot;a)$ we define
\begin{equation}\label{J1}
\begin{split}
J(r)=J(r;a) & :=\frac{1}{2}\alpha(r)g'(r;a)^2+\beta(r)g(r;a)g'(r;a) \\
& \quad +\frac{1}{2}\gamma(r)g(r;a)^2+\frac{p-1}{p}\delta(r)g(r;a)^{p/(p-1)}
\end{split}
\end{equation}
for $r\in (0,R_1(a))$, where $\alpha$,
$\beta$, $\gamma$, and $\delta$ are functions to be
determined later. We proceed as in \cite{ShWa16} to look for a
functional $J$ solving a differential equation of the form
$J'=Gg^2$ for some function $G$. Calculating $J'$ and using \eqref{ODEg} in order to replace $g''$ in the calculations, we find, for $r\in (0,R_1(a))$,
\begin{equation*}
\begin{split}
J'(r)&=g'(r)^2 \left[\frac{1}{2}\alpha'(r)-\left(1+\frac{N-1}{r}\right)\alpha(r)+\beta(r)\right]\\
&+g(r)g'(r)\left[\frac{N-1}{r^2}\alpha(r)+\beta'(r)-\left(1+\frac{N-1}{r}\right)\beta(r)+\gamma(r)\right]\\
&+g(r)^{1/(p-1)}g'(r)\left[-\alpha(r)+\delta(r)\right]+g(r)^{p/(p-1)}\left[-\beta(r)+\frac{p-1}{p}\delta'(r)\right]\\
&+g(r)^2\left[\frac{N-1}{r^2}\beta(r)+\frac{1}{2}\gamma'(r)\right].
\end{split}
\end{equation*}
The idea is then to choose the functions $\alpha$, $\beta$, $\gamma$, and $\delta$ in order to vanish all the coefficients in the expression above for $J$, except for the one in front of $g^2$.
More specifically we require
\begin{eqnarray}
&&\beta(r)=-\frac{1}{2}\alpha'(r)+\left(1+\frac{N-1}{r}\right)\alpha(r),\label{J2}\\
&&\gamma(r)=-\beta'(r)+\left(1+\frac{N-1}{r}\right)\beta(r)-\frac{N-1}{r^2}\alpha(r),\label{J3}\\
&&\delta(r)=\alpha(r),\label{J4}\\
&&\frac{p-1}{p}\delta'(r)=\beta(r),\label{J5}
\end{eqnarray}
in order to have
\begin{equation}\label{J6}
J'(r)=G(r)g(r)^2, \quad
G(r):=\frac{N-1}{r^2}\beta(r)+\frac{1}{2}\gamma'(r)\ , \qquad r\in (0,R_1(a))\ .
\end{equation}
Combining \eqref{J2}, \eqref{J4}, and \eqref{J5} and replacing all
the unknown functions in terms of $\delta$, we obtain
$$
\frac{\delta'(r)}{\delta(r)}=\frac{2p}{3p-2}\frac{\varrho'(r)}{\varrho(r)}, \qquad r>0\ ,
$$
so that $\delta=\varrho^{2p/(3p-2)}$. Thanks to \eqref{J2} and \eqref{J4}, we find
\begin{equation}\label{J7}
\alpha(r)=\delta(r)=\varrho(r)^{2p/(3p-2)}, \quad
\beta(r)=\frac{2(p-1)}{3p-2}\left(1+\frac{N-1}{r}\right)\varrho(r)^{2p/(3p-2)}.
\end{equation}
We next use \eqref{J3} and we obtain after rather tedious, but
straightforward calculations
\begin{equation}\label{J8}
\begin{split}
\gamma(r)&= - \left[\frac{2(p-1)(2-p)}{(3p-2)^2}+\frac{4(N-1)(p-1)(2-p)}{(3p-2)^2}\frac{1}{r}\right.\\
&\left.+\frac{N-1}{(3p-2)^2}[p(3p-2)+2(N-1)(p-1)(2-p)]\frac{1}{r^2}\right]\varrho(r)^{2p/(3p-2)}.
\end{split}
\end{equation}
Replacing the formulas giving the functions $\beta$ and $\gamma$ from \eqref{J7}
and \eqref{J8} into \eqref{J6} and again after rather long
calculations, one finds that $G/\alpha$ is a cubic polynomial in $1/r$. More precisely,
\begin{equation}\label{J9}
\frac{(3p-2)^3}{p}\frac{G(r)}{\alpha(r)}= P\left( \frac{1}{r} \right)\ , \qquad r>0\ ,
\end{equation}
where $P$ is a the cubic polynomial
$$
P(z) :=  M_3(p) (N-1) z^3 +M_2(p) (N-1) z^2 + M_1(p) (N-1) z+ M_0(p),
$$
its coefficients being given by
\begin{eqnarray*}
&&M_3(p):=(3p-2)^2+(N-1)(3p-2)(3p-4)-2(p-1)(2-p)(N-1)^2,\\
&&M_2(p):=(3p-2)(3p-4)-6(N-1)(p-1)(2-p),\\
&&M_1(p):=-6(p-1)(2-p), \quad M_0(p):=-2(p-1)(2-p).
\end{eqnarray*}

We emphasize here that \textsl{the function $G$ does not depend on the parameter $a$.}

The following result will be essential for the proof of Theorem~\ref{th.uniq}.

%%%%%%%%%%%%%%%%%%%%
\begin{lemma}\label{lemaG}
If $p\in(p_c,2)$, there exists $r_G>0$ such that $G(r)>0$ if
$r\in[0,r_G)$ and $G(r)<0$ if $r>r_G$.
\end{lemma}
%%%%%%%%%%%%%%%%%%%%

\begin{proof}
From the definitions above, we easily notice that $M_0(p)<0$ and
$M_1(p)<0$, since $p\in (1,2)$. We will next prove that $M_3(p)>0$
for $p\in(p_c,2)$. To this end, we first note that
\begin{equation}\label{J10}
M_3(p_c)= 4 \frac{N^3}{(N+1)^2}>0.
\end{equation}
For $p>p_c$, we compute $M_3'(p)$ and find
\begin{equation*}
\begin{split}
M_3'(p)&=(4N^2+10N+4)p-6N(N+1) \\
& > (4N^2+10N+4)\frac{2N}{N+1}-6N(N+1)\\
& = \frac{2N}{N+1}\left(N^2+4N+1\right)>0.
\end{split}
\end{equation*}
Thus, $M_3$ is increasing on $(p_c,\infty)$ and, taking into account \eqref{J10},
we conclude that $M_3(p)>0$ for $p>p_c$. Therefore, denoting the (possibly complex) roots of the cubic polynomial $P$ by $z_1$, $z_2$, and $z_3$ with $z_1\in\real$ and using both the relations between roots and coefficients
$$
z_1 z_2 z_3=-\frac{M_0(p)}{(N-1)M_3(p)}>0, \quad
z_1 z_2 + z_1 z_3 + z_2 z_3 = \frac{M_1(p)}{M_3(p)}<0\ ,
$$
and the properties
$$
P(0)= M_0(p)<0 \ , \qquad \lim_{z\to\infty} P(z) = \infty\ ,
$$
we realize that the polynomial $P$ has exactly one positive root, which we
denote by $1/r_G$. It is then easy to check that $G(r)>0$ if $r\in(0,r_G)$ and $G(r)<0$ if $r>r_G$, completing the proof.
\end{proof}

Combining \eqref{J6} and Lemma~\ref{lemaG} provides interesting properties of the function $J(\cdot;a)$ defined in \eqref{J1} which we summarize in the next proposition.

%%%%%%%%%%%%%%%%%%%%
\begin{proposition}\label{propJ}
Let $a>0$. Then $J(0;a)=0$ and $J(\cdot;a)$ is increasing on $\left( 0,\min\{r_G,R_1(a)\} \right)$. If $R_1(a)>r_G$ then $J(\cdot;a)$ is decreasing on $\left( r_G , R_1(a) \right)$.
\end{proposition}
%%%%%%%%%%%%%%%%%%%%

\begin{proof}
According to \eqref{ICg} and \eqref{J1}, we have as $r\to 0$,
\begin{equation}\label{interm23}
J(r)\sim\left[\frac{1}{2}\alpha(r)+r\beta(r)+\frac{1}{2}r^2\gamma(r)\right]\left(\frac{a}{N}\right)^2+\frac{p-1}{p}\alpha(r)r^{p/(p-1)}\left(\frac{a}{N}\right)^{p/(p-1)}.
\end{equation}
Recalling the formulas for $\alpha$, $\beta$, and $\gamma$ in \eqref{J7} and \eqref{J8} we have
$$
r\beta(r)\sim\frac{2(p-1)}{3p-2}(N-1+r)\alpha(r)\to0 \quad {\rm as}
\ r\to0,
$$
and
$$
r^2\gamma(r)\sim-\frac{N-1}{(3p-2)^2}\left[p(3p-2)+2(N-1)(p-1)(2-p)\right]\alpha(r)\to0
\quad {\rm as} \ r\to 0.
$$
Combining these properties with \eqref{interm23} gives $J(0)=0$. The
monotonicity properties of $J$ readily follow from \eqref{J6} and
Lemma~\ref{lemaG}.
\end{proof}

The usefulness of the Pohozaev functional $J$ becomes clear from the
next two lemmas. Indeed, with its help we provide a sharp difference
between the sets $B$ and $C$, which is one of the last technical
steps towards uniqueness of the profile in $B$.

%%%%%%%%%%%%%%%%%%%%
\begin{lemma}\label{lem48}
Let $a\in B$. Then $$J(r;a)>0 \ {\rm for} \ r\in(0,\infty), \quad
\lim\limits_{r\to\infty}J(r;a)=0.$$
\end{lemma}
%%%%%%%%%%%%%%%%%%%%

\begin{proof}
Let $a\in B$. We first notice from \eqref{J1}, \eqref{J7}, and
\eqref{J8} that
\begin{equation}\label{J11}
J(r)=\alpha(r)g(r)^2
\left[\frac{1}{2}\left(\frac{g'(r)}{g(r)}\right)^2+\frac{\beta(r)}{\alpha(r)}\frac{g'(r)}{g(r)}+\frac{\gamma(r)}{2\alpha(r)}+\frac{p-1}{p}g(r)^{(2-p)/(p-1)}\right],
\end{equation}
with
\begin{equation}\label{J12}
\lim\limits_{r\to\infty}\frac{\beta(r)}{\alpha(r)}=\frac{2(p-1)}{3p-2},
\qquad
\lim\limits_{r\to\infty}\frac{\gamma(r)}{2\alpha(r)}=-\frac{(2-p)(p-1)}{(3p-2)^2}.
\end{equation}
Since $a\in B$, it follows from \eqref{limitw} and Lemma~\ref{lem47} that
$$
g(r)\sim \frac{l(a)}{\varrho(r)} \;\;\text{ as }\;\; r\to\infty\ , \qquad \lim\limits_{r\to\infty} \frac{g'(r)}{g(r)} = - 1 + \lim\limits_{r\to\infty} \frac{w'(r)}{w(r)} = -1\ .
$$
Therefore, we infer from \eqref{J7}, \eqref{J11}, and \eqref{J12} that, as $r\to\infty$,
\begin{equation*}
\begin{split}
J(r)&\sim\varrho(r)^{2p/(3p-2)-2}l(a)^2\left[\frac{1}{2}-\frac{2(p-1)}{3p-2}-\frac{(p-1)(2-p)}{(3p-2)^2}\right]\\
&=l(a)^2\varrho(r)^{-4(p-1)/(3p-2)}\frac{p(2-p)}{2(3p-2)^2},
\end{split}
\end{equation*}
hence $\lim\limits_{r\to\infty}J(r)=0$. Since $J$ is increasing on
$(0,r_G)$ and decreasing on $(r_G,\infty)$ with $J(0)=0$ by
Proposition~\ref{propJ}, we conclude that $J(r)>0$ for
$r\in(0,\infty)$.
\end{proof}

%%%%%%%%%%%%%%%%%%%%
\begin{lemma}\label{lem49}
Let $a>0$. Then $a\in C$ if and only if there exists $\bar{r}\in (0,R_1(a))$ such that $J(\bar{r};a)<0$. 

Moreover, if $a\in C$, then
$$
\lim\limits_{r\to\infty} J(r;a)=-\infty.
$$
\end{lemma}
%%%%%%%%%%%%%%%%%%%%

\begin{proof}
Consider first $a\in C$. By Corollary~\ref{cor.1}, we deduce that
$$
g(r)\sim\left(\frac{2-p}{p-1}r\right)^{-(p-1)/(2-p)} \quad {\rm as}
\ r\to\infty,
$$
and
$$
\frac{g'(r)}{g(r)}=\frac{w'(r)}{w(r)}-1-\frac{N-1}{r}\to 0 \quad
{\rm as} \ r\to\infty.
$$
We then obtain from \eqref{J11} and \eqref{J12} that
$$
J(r)\sim-\frac{(p-1)(2-p)}{(3p-2)^2}\varrho(r)^{2p/(3p-2)}\left(\frac{2-p}{p-1}r\right)^{-2(p-1)/(2-p)}\quad
{\rm as} \ r\to\infty,
$$
from which we deduce that $J(r)\to-\infty$ as $r\to\infty$ and thus takes negative values.

Conversely, let $a>0$ be such that there is $\bar{r}\in (0,R_1(a))$ with $J(\bar{r};a)<0$. Assume for contradiction that $R_1(a)<\infty$. On the one hand it follows from Proposition~\ref{propJ} that $\bar{r}>r_G$ and $J(r;a) \le J(\bar{r};a)<0$ for $r\in (\bar{r},R_1(a))$. On the other hand $J(R_1(a);a) = \alpha(R_1(a))
g'(R_1(a);a)^2 \ge 0$ by \eqref{r1a} and \eqref{J1}, which contradicts the previous statement. Consequently, $R_1(a)=\infty$, so that $a\in B\cup C$ and Lemma~\ref{lem48} excludes that $a$ belongs to $B$.
\end{proof}

A by-product of Lemma~\ref{lem48} and Lemma~\ref{lem49} is an alternative characterization of the sets $B$ and $C$.

%%%%%%%%%%%%%%%%%%%%
\begin{corollary}\label{corrmk}
Let $a\in B\cup C$. Then
\begin{eqnarray*}
&&a\in B \ {\rm if \ and \ only \ if } \
\lim\limits_{r\to\infty}J(r;a)=0,\\
&&a\in C \ {\rm if \ and \ only \ if } \
\lim\limits_{r\to\infty}J(r;a)=-\infty.
\end{eqnarray*}
\end{corollary}
%%%%%%%%%%%%%%%%%%%%

The previous characterization of the set $C$ via the properties of
$J(\cdot;a)$ allows us to show that the set $C$ is non-empty.

%%%%%%%%%%%%%%%%%%%%
\begin{lemma}\label{lemaC}
The set $C$ is non-empty and open. Moreover, there exists $a_0>0$
such that $(0,a_0)\subseteq C$.
\end{lemma}
%%%%%%%%%%%%%%%%%%%%

\begin{proof}
We prove first that $C$ is an open set. Consider $\bar{a}\in C$. We
deduce from Corollary~\ref{corrmk} that there exists $R>r_G$ such
that $J(R;\bar{a})<-2$. By continuous dependence, there exists
$\varepsilon>0$ such that
$$
R_1(a)>2R \;\;\text{ and }\;\; J(R;a)<-1 \quad {\rm for} \ a\in(\bar{a}-\varepsilon , \bar{a}
+\varepsilon),
$$
recalling that $R_1(a)$ is defined in \eqref{r1a}. Then Lemma~\ref{lem49} guarantees that
$(\bar{a}-\varepsilon,\bar{a}+\varepsilon)\subset C$.

Proving that $C$ is non-empty is more involved but also relies on a continuous dependence argument. For $a>0$ and $r\in (0,R_1(a))$, define
$$
z(r;a):=\frac{1}{a}g(r;a), \quad Z(r;a):=\frac{1}{a^2}J(r;a)\ .
$$
 It follows from \eqref{ODEg} that
\begin{equation}\label{ODEz}
z''(r;a)+\left(1+\frac{N-1}{r}\right)z'(r;a)-\frac{N-1}{r^2}z(r;a)+a^{(2-p)/(p-1)}z(r;a)^{1/(p-1)}=0
\end{equation}
for $r\in (0,R_1(a))$, with initial conditions
$$
z(0;a)=0, \quad z'(0;a)=\frac{1}{N}.
$$
Since $p\in (1,2)$ the nonlinear term in \eqref{ODEz} vanishes in the limit $a\to 0$ and we introduce the solution $z_0$ to the limit problem:
\begin{eqnarray}
&&z_0''(r)+\left(1+\frac{N-1}{r}\right)z_0'(r)-\frac{N-1}{r^2}z_0(r)=0, \quad r>0, \label{ODEz0}\\
&&z_0(0)=0, \quad z_0'(0)=\frac{1}{N}.\label{ICz0}
\end{eqnarray}
In fact,
$$
z_0(r)=\frac{1}{\varrho(r)}\int_0^r\varrho(s)\,ds\ , \qquad r>0\ .
$$
Indeed, by a simple calculation,
$$
z_0'(r)=1-\left(1+\frac{N-1}{r}\right)z_0(r), \quad
z_0''(r)=-\left(1+\frac{N-1}{r}\right)z_0'(r)+\frac{N-1}{r^2}z_0(r),
$$
whence \eqref{ODEz0} is fulfilled, while l'Hospital rule ensures
that
$$
\lim\limits_{r\to0}\frac{z_0(r)}{r}=\lim\limits_{r\to0}\frac{\varrho(r)}{r\varrho'(r)+\varrho(r)}=\lim\limits_{r\to0}\frac{1}{r+N-1+1}=\frac{1}{N},
$$
leading readily to \eqref{ICz0}. In addition, by \eqref{asymptrho} and the l'Hospital rule,
\begin{equation}
\lim\limits_{r\to\infty}z_0(r)=\lim\limits_{r\to\infty}\frac{\varrho(r)}{\varrho'(r)}=1\ . \label{Y1}
\end{equation}
Since $z(\cdot;a)\to z_0$ as $a\to 0$ uniformly on compact subsets
of $(0,\infty)$, we infer from that convergence, the positivity of
$z_0$ in $(0,\infty)$, and \eqref{Y1} that
\begin{equation}
\lim_{a\to\infty} R_1(a) = \infty\ . \label{Y2}
\end{equation}

We also introduce the (formal) limit $Z_0$ of the Pohozaev functional $Z(\cdot;a)$ as $a\to 0$,
\begin{equation}\label{J13}
\begin{split}
Z_0(r)&:=\frac{1}{2}\alpha(r)z_0'(r)^2+\beta(r)z_0'(r)z_0(r)+\frac{1}{2}\gamma(r)z_0(r)^2\\
&=\alpha(r)z_0(r)^2\left[\frac{\gamma(r)}{2\alpha(r)}+\frac{\beta(r)}{\alpha(r)}\frac{z_0'(r)}{z_0(r)}+\frac{1}{2}\left(\frac{z_0'(r)}{z_0(r)}\right)^2\right],
\end{split}
\end{equation}
with $\alpha$, $\beta$, and $\gamma$ given by \eqref{J7} and \eqref{J8} and claim that
\begin{equation}\label{J14}
\lim\limits_{r\to\infty} Z_0(r)=-\infty.
\end{equation}
Indeed, since
$$
\frac{z_0'(r)}{z_0(r)}=\frac{1}{z_0(r)}-\left(1+\frac{N-1}{r}\right) \ ,
$$
we deduce from \eqref{Y1} that
$$
\lim\limits_{r\to\infty}\frac{z_0'(r)}{z_0(r)}=0.
$$
Thus, taking into account that
$$
\alpha(r)z_0^2(r)\sim\alpha(r)\to\infty \quad {\rm as} \ r\to\infty
$$
and recalling the limits in \eqref{J12}, we obtain \eqref{J14}. It
follows from \eqref{Y2} and \eqref{J14} that there exist $R_0>r_G$
and $a_0>0$ such that
$$
Z_0(R_0)<-2 \;\;\text{ and }\;\; R_1(a)>2R_0 \;\;\text{ for }\;\; a
\in (0,a_0)\ .
$$
Then $Z(\cdot;a)$ is well-defined on $(0,2R_0)$ for $a\in (0,a_0)$
and, since $z(\cdot;a)\to z_0$ in $C^1([r_G,2R_0])$ as $a\to 0$ by
continuous dependence, we find
$$
\lim\limits_{a\to 0} Z(R_0;a)=Z_0(R_0)<-2\ ,
$$
and we may assume (possibly taking a smaller value of $a_0$)  that
\begin{equation*}
Z(R_0;a) \le -1\ , \qquad a\in (0,a_0)\ .
\end{equation*}
Consequently, for $a\in (0,a_0)$,
\begin{equation*}
J(R_0;a) = a^2 Z(R_0;a) \le -a^2 < 0\ ,
\end{equation*}
and Lemma~\ref{lem49} readily entails that $a\in C$. We have thus proved that $(0,a_0)\subseteq C$.
\end{proof}

%%%%%%%%%%%%%%%%%%%%
%%%%%%%%%%%%%%%%%%%%
\subsection{Uniqueness of the fast decaying profile}\label{s4.4}
%%%%%%%%%%%%%%%%%%%%
%%%%%%%%%%%%%%%%%%%%

In this subsection we complete the proof of Theorem~\ref{th.uniq} and proceed as in the proofs of \cite[Lemma~3 \& Proposition~3]{ShWa16}. To this end, we need two preparatory technical results.

%%%%%%%%%%%%%%%%%%%%
\begin{lemma}\label{lem410}
Let $0<a_1<a_2<\infty$ such that $a_i\in B\cup C$, $i=1,2$. We define the Wronskian $W=g'(\cdot;a_1) g(\cdot;a_2) - g(\cdot;a_1) g'(\cdot;a_2)$ of the solutions $g(\cdot;a_1)$ and $g(\cdot;a_2)$. Then
\begin{equation}\label{Wronsk}
W(r)=\int_0^r\frac{\varrho(s)}{\varrho(r)}\left(g(s;a_2)^{(2-p)(p-1)}-g(s;a_1)^{(2-p)/(p-1)}\right)g(s;a_1) g(s;a_2)\,ds
\end{equation}
for $r>0$.
\end{lemma}
%%%%%%%%%%%%%%%%%%%%

\begin{proof}
We set $g_i:=g(\cdot;a_i)$, $i=1,2$.
By \eqref{ODEg} and direct calculations we get
\begin{equation*}
\begin{split}
W'(r)&=g_2(r)\left[-\frac{\varrho'(r)}{\varrho(r)}g_1'(r)-g_1(r)^{1/(p-1)}+\frac{N-1}{r^2}g_1(r)\right]\\
& \qquad - g_1(r)\left[-\frac{\varrho'(r)}{\varrho(r)}g_2'(r)-g_2(r)^{1/(p-1)}+\frac{N-1}{r^2}g_2(r)\right]\\
&=-\frac{\varrho'(r)}{\varrho(r)}W(r)-g_1(r)^{1/(p-1)}g_2(r)+g_2(r)^{1/(p-1)}g_1(r),
\end{split}
\end{equation*}
hence
\begin{align*}
(W\varrho)'(r) &= \varrho(r)\left[W'(r)+\frac{\varrho'(r)}{\varrho(r)}W(r)\right] \\
& =\varrho(r)g_1(r)g_2(r)\left[g_2(r)^{(2-p)/(p-1)}-g_1(r)^{(2-p)/(p-1)}\right].
\end{align*}
Since $(W\varrho)(0)=W(0)=0$, \eqref{Wronsk} follows by direct
integration.
\end{proof}

%%%%%%%%%%%%%%%%%%%%
\begin{lemma}\label{lem411}
Let $0<a_1<a_2<\infty$ such that $a_i\in B\cup C$, $i=1,2$. If $J(r;a_1)\geq0$ for $r\in(0,\infty)$, then
\begin{equation}\label{compar}
\frac{d}{dr} \left(\frac{g(r;a_2)}{g(r;a_1)}\right)<0, \quad r\in(0,\infty).
\end{equation}
\end{lemma}
%%%%%%%%%%%%%%%%%%%%

\begin{proof}
We set $g_i:=g(\cdot;a_i)$ and $J_i := J(\cdot;a_i)$, $i=1,2$, where $J(\cdot;a)$ is defined in \eqref{J1}. Introducing $q:=g_2/g_1$, we notice that
\begin{equation}
\lim\limits_{r\to0}q(r)=\lim\limits_{r\to0}\frac{g_2(r)}{g_1(r)}=\lim\limits_{r\to0}\frac{g_2'(r)}{g_1'(r)}=\frac{a_2}{a_1}>1\ . \label{Y5}
\end{equation}
Next, recalling that $W$ is defined in Lemma~\ref{lem410}, we obtain, for $r>0$,
\begin{eqnarray}
q'(r) & = & - \frac{W(r)}{g_1(r)^2} \nonumber\\
& =& \frac{1}{g_1^2(r)}\int_0^r\frac{\varrho(s)}{\varrho(r)}\left(g_1(s)^{(2-p)(p-1)}-g_2(s)^{(2-p)/(p-1)}\right)g_1(s)g_2(s)\,ds. \label{interm22}
\end{eqnarray}
Since $a_1<a_2$, it follows from \eqref{ICg} that $0<g_1(r)<g_2(r)$
for $r\in(0,\varepsilon)$ for some $\varepsilon>0$ sufficiently small and thus
$q'<0$ in $(0,\varepsilon)$ by \eqref{interm22}. Define then
$$
r_*:=\inf\{r>0: q'(r)>0\} > \varepsilon >0\ ,
$$
and assume for contradiction that $r_*<\infty$. Then $q'(r)<0$ for
$r\in(0,r_*)$ and $q'(r_*)=0$. In particular, $q$ is decreasing in
$(0,r_*)$. We further claim that
\begin{equation}
q(r_*)<1\ . \label{Y6}
\end{equation}
Indeed, if this is not true, there
holds $q(r)>1$ for $r\in(0,r_*)$, whence $g_1(r)<g_2(r)$ for
$r\in(0,r_*)$ and thus $q'(r_*)<0$ by \eqref{interm22}, which is a contradiction with the definition of $r_*$.

We next introduce
$$
X(r):=q^2(r)J_1(r)-J_2(r), \qquad r>0\ .
$$
Thanks to the definition of $q$, it follows from \eqref{J6} that, for $r\in(0,r_*)$,
\begin{equation}\label{X1}
X'(r)= 2q(r)q'(r)J_1(r)<0.
\end{equation}
In addition, we recall that $J_i(0)=0$ by Proposition~\ref{propJ} which, together with \eqref{Y5}, leads us to \begin{equation}\label{X2}
\lim\limits_{r\to0}X(r)=0.
\end{equation}
Expanding now $X$ in terms of
$\alpha$, $\beta$, $\gamma$, and $\delta$, we find
\begin{equation*}
\begin{split}
X&=\frac{\alpha}{2}\left[ q^2 (g_1')^2-(g_2')^2 \right]+ \beta\left(q^2g_1g_1'-g_2g_2'\right) +\frac{\gamma}{2} \left(q^2g_1^2-g_2^2\right)\\
&+\frac{p-1}{p}\delta \left(q^2g_1^{p/(p-1)}-g_2^{p/(p-1)}\right)\\
&=\frac{\alpha}{2g_1^2}\left[ g_2^2 (g_1')^2-g_1^2 (g_2')^2\right]+\frac{\beta g_2}{g_1}\left(g_2g_1'-g_1g_2'\right)\\
&+\frac{p-1}{p}\delta g_2^2\left(g_1^{(2-p)/(p-1)}-g_2^{(2-p)/(p-1)}\right)\\
&=\left[\frac{\alpha}{2g_1^2}\left(g_2g_1'+g_1g_2'\right)+\frac{\beta g_2}{g_1}\right]W\\
&+\frac{p-1}{p}\delta g_2^2\left(g_1^{(2-p)/(p-1)}-g_2^{(2-p)/(p-1)}\right).
\end{split}
\end{equation*}
Evaluating $X$ at $r=r_*$ and taking into account that
$W(r_*)=-q'(r_*)/g_1(r_*)^2=0$, we conclude that
\begin{equation*}
\begin{split}
X(r_*)&=\frac{p-1}{p}\alpha(r_*)g_2(r_*)^2\left(g_1(r_*)^{(2-p)/(p-1)}-g_2(r_*)^{(2-p)/(p-1)}\right)\\
&=\frac{p-1}{p}\alpha(r_*)g_2(r_*)^2g_1(r_*)^{(2-p)/(p-1)}\left(1-q(r_*)^{(2-p)/(p-1)}\right)>0,
\end{split}
\end{equation*}
since $q(r_*)<1$ by \eqref{Y6}. This contradicts
\eqref{X1} and \eqref{X2} and thus, $r_*=\infty$, completing the proof.
\end{proof}

With all these previous steps, we are now in a position to end up the proof of Theorem~\ref{th.uniq}.

\begin{proof}[Proof of Theorem~\ref{th.uniq}]
We infer from Corollary~\ref{corA} and Lemma~\ref{lemaC} that $B$ is
a non-empty set, as $A$ and $C$ are both open and non-empty. We next
show that $B$ is in fact a singleton. Indeed, assume for
contradiction that there exist $a_1$, $a_2\in B$ such that
$a_1<a_2$. Setting $g_i=g(\cdot;a_i)$ and $J_i:= J(\cdot;a_i)$,
$i=1,2$, as well as $q=g_2/g_1$ and $X=q^2J_1-J_2$ as before, it
follows from Lemma~\ref{lem48} that $J_i(r)>0$ for $r\in (0,\infty)$
and $i=1,2$, and
\begin{equation}
\lim\limits_{r\to\infty}J_1(r)=\lim\limits_{r\to\infty}J_2(r)=0. \label{Y10}
\end{equation}
Thanks to the just mentioned positivity of $J_1$, we apply Lemma~\ref{lem411} and obtain that $q'<0$ in
$(0,\infty)$. Consequently
\begin{equation}
X'(r)=2q(r)q'(r)J_1(r)<0, \quad r\in(0,\infty). \label{Y11}
\end{equation}
Since both $a_1$ and $a_2$ belong to $B$, we also know from \eqref{limitw} that there exists $l(a_i)\in(0,\infty)$ such that  $g_i(r)\sim l(a_i)\varrho(r)^{-1}$ as $r\to\infty$, $i=1,2$, so that $q(r)\to l(a_2)/l(a_1)$ as $r\to\infty$. Combining this fact with \eqref{Y10} gives
$$
\lim\limits_{r\to\infty} X(r)=0.
$$
Recalling \eqref{Y11} we deduce that $X$ is positive and decreasing on $(0,\infty)$, which contradicts \eqref{X2}.

Therefore $B$ is a singleton and we denote the only element in $B$
by $a_*$ . We then infer from Lemma~\ref{lem44}, Corollary~\ref{corA}, and Lemma~\ref{lemaC} that $C=(0,a_*)$, $B=\{a_*\}$, and
$A=(a_*,\infty)$. The corresponding asymptotic behavior for $f(r;a)$
as $r\to\infty$ when $a\in C$ follows now readily from
Corollary~\ref{cor.2}, \eqref{f8}, and the definition of $w$.
Finally, for $a=a_*$, according to \eqref{limitw}, $g(r;a_*)\sim l(a_*)/\varrho(r)$ as $r\to\infty$. The behavior of $f(r;a_*)$
as $r\to\infty$ claimed in \eqref{decay}, with
$$
c_*=(p-1)l(a_*)^{1/(p-1)},
$$
is then obtained by integration and using Lemma~\ref{lem41}~(d).
\end{proof}

%%%%%%%%%%%%%%%%%%%%
\begin{remark}\label{rmkvar}
The existence of profiles $f(\cdot;a)$ with $a\in B$ may also be established with the help of the  variational structure of \eqref{eq1} used in Section~\ref{sec.s3}.  More precisely, such a profile $f$ can be found as a minimizer of the constrained problem:
$$
\inf\left\{\int_0^{\infty} \varrho(r)\frac{|f'(r)|^p}{p}\,dr:
\int_0^{\infty} \varrho(r)\frac{|f(r)|^2}{2}\,dr=1\right\},
$$
but this approach gives no clue about the uniqueness issue.
\end{remark}
%%%%%%%%%%%%%%%%%%%%

%%%%%%%%%%%%%%%%%%%%
%%%%%%%%%%%%%%%%%%%%
\section*{Acknowledgments} R. G. I. is supported by the
Severo Ochoa Excellence project SEV-2015-0554 (MINECO, Spain). Part
of the work has been completed while R. G. I. was enjoying a
one-month ``Invited Professor" stay at the Institut de Mathématiques
de Toulouse, and he thanks for the hospitality and support.
%%%%%%%%%%%%%%%%%%%%
%%%%%%%%%%%%%%%%%%%%

%%%%%%%%%%%%%%%%%%%%
%%%%%%%%%%%%%%%%%%%%
\bibliographystyle{acm}
\bibliography{CriticalExtinction}
%%%%%%%%%%%%%%%%%%%%
%%%%%%%%%%%%%%%%%%%%

%%%%%%%%%%%%%%%%%%%%
%%%%%%%%%%%%%%%%%%%%
\end{document}